\documentclass[11pt, a4paper,leqno]{amsart}
\usepackage{amsmath,amsthm,amscd,amssymb,amsfonts, amsbsy}
\usepackage{latexsym}
\usepackage{txfonts}
\usepackage{exscale}
\usepackage{bbm}
\usepackage{enumitem}
\usepackage{graphicx}
\usepackage{soul}
\usepackage[colorlinks,citecolor=red,pagebackref,hypertexnames=false]{hyperref}
\usepackage{color}

\def\Yint#1{\mathchoice
    {\YYint\displaystyle\textstyle{#1}}%
    {\YYint\textstyle\scriptstyle{#1}}%
    {\YYint\scriptstyle\scriptscriptstyle{#1}}%
    {\YYint\scriptscriptstyle\scriptscriptstyle{#1}}%
      \!\iint}
\def\YYint#1#2#3{{\setbox0=\hbox{$#1{#2#3}{\iint}$}
    \vcenter{\hbox{$#2#3$}}\kern-.51\wd0}}
\def\longdash{{-}\mkern-3.5mu{-}} 
\def\tiltlongdash{\rotatebox[origin=c]{15}{$\longdash$}}
\def\fiint{\Yint\tiltlongdash}

\calclayout
\allowdisplaybreaks

\theoremstyle{plain}
\newtheorem{theorem}[equation]{Theorem}
\newtheorem{lemma}[equation]{Lemma}
\newtheorem{corollary}[equation]{Corollary}

\theoremstyle{definition}
\newtheorem{definition}[equation]{Definition}

\theoremstyle{remark}
\newtheorem{remark}[equation]{Remark}

\newtheorem{claim}[equation]{Claim}

\numberwithin{equation}{section}

\DeclareMathOperator{\Id}{Id}

\newcommand{\RR}{{\mathbb{R}}}
\newcommand{\pO}{\partial\Omega}

\newcommand{\re}{\mathbb{R}}
\newcommand{\rn}{\mathbb{R}^n}

\newcommand{\ree}{\mathbb{R}^{n+1}}

\newcommand{\om}{\Omega}

\def\XXint#1#2#3{{\setbox0=\hbox{$#1{#2#3}{\int}$ }
\vcenter{\hbox{$#2#3$ }}\kern-.6\wd0}}

\newcommand{\hm}{\omega}

\newcommand{\RNum}[1]{\uppercase\expandafter{\romannumeral #1\relax}}

\def\Yint#1{\mathchoice
    {\YYint\displaystyle\textstyle{#1}}%
    {\YYint\textstyle\scriptstyle{#1}}%
    {\YYint\scriptstyle\scriptscriptstyle{#1}}%
    {\YYint\scriptscriptstyle\scriptscriptstyle{#1}}%
      \!\iint}
\def\YYint#1#2#3{{\setbox0=\hbox{$#1{#2#3}{\iint}$}
    \vcenter{\hbox{$#2#3$}}\kern-.51\wd0}}
\def\longdash{{-}\mkern-3.5mu{--}} 
\def\tiltlongdash{\rotatebox[origin=c]{15}{$\longdash$}}
\def\fiint{\Yint\tiltlongdash}

\renewcommand{\emptyset}{\mbox{\textup{\O}}}

\DeclareMathOperator{\supp}{supp}

\DeclareMathOperator{\divg}{div}

\def\div{\mathop{\operatorname{div}}\nolimits}

\begin{document}
\allowdisplaybreaks

\author{S. Bortz}
\address{Department of Mathematics, University of Alabama, Tuscaloosa, AL, 35487, USA}
	\email{sbortz@ua.edu}

\author{T. Toro}
\address{Department of Mathematics, University of Washington, Seattle, WA
98195, USA}
\email{toro@uw.edu}

\author{Z. Zhao}
\address{Department of Mathematics, University of Chicago, Chicago, IL 60637, USA}
\email{zhaozh@uchicago.edu}

\title[Elliptic measures for DKP operators]{Elliptic measures for Dahlberg-Kenig-Pipher operators: Asymptotically optimal estimates}

\begin{abstract}
Questions concerning quantitative and asymptotic properties of the elliptic measure corresponding to a uniformly elliptic divergence form operator have been the focus of recent studies.
In this setting we show that the elliptic measure of an operator with coefficients satisfying a vanishing Carleson condition in the upper half space is an asymptotically optimal $A_\infty$ weight. In particular, for such operators the logarithm of the elliptic kernel is in the space of (locally) vanishing mean oscillation. 

To achieve this, we prove local, quantitative estimates on a quantity (introduced by Fefferman, Kenig and Pipher) that controls the $A_\infty$ constant. Our work uses recent results obtained by David, Li and Mayboroda. These quantitative estimates may offer a new framework to approach similar problems.
\end{abstract}

\thanks{T.T. was partially supported by the Craig McKibben \& Sarah Merner Professor in Mathematics, and by NSF grant DMS-1954545. Z.Z. was partially supported by NSF grant DMS-1902756.}

\maketitle
\tableofcontents

\section{Introduction}
In this article we investigate the qualitative and quantitative properties of elliptic measures $\hm_L$ associated to divergence form uniformly elliptic operators $L = -\div A \nabla$ with certain variable coefficients in the half space $\mathbb{R}^{n+1}$. The coefficients satisfy the so-called weak Dahlberg-Kenig-Pipher (DKP) condition, which is a Carleson measure condition on the $L^2$-oscillation of the coefficients on Whitney regions (see Definition \ref{DKPconditions.def}). A very closely related condition was introduced by Dahlberg. 
It shown to be sufficient for $L^p$ solvability of the Dirichlet problem for some $p >1$ by Kenig and Pipher in \cite{KP} (see also \cite{HL}). Subsequently, it was shown by Dindos, Petermichl and Pipher \cite{DPP} that sufficient smallness in a similar\footnote{See Section \ref{kpdpp.sect}.} DKP-type condition allows one to solve the $L^p$-Dirichlet problem for $p >1$ close to $1$ (the `smallness' depends on $p$).

We are particularly interested in the endpoint case when the Carleson norm defining the DKP condition has `vanishing trace'. Under this assumption we study the $BMO$ norm of the logarithm of the elliptic kernel, $k_L$ at small scales (here $k_L = \tfrac{d\hm_L}{dx}$). For any non-negative locally integrable function $w$ if $\log w$ has small norm $BMO$ norm then, roughly speaking, $w$ is `almost' an `optimal' $A_\infty$ weight. The connection between the space of BMO (or $VMO$) and $A_\infty$ (Muckenhoupt) weights is well documented \cite{GCRDF, Sarason, Kor, Kor2} as is the connection between the $A_\infty$ condition for the elliptic measure and the solvability of an $L^p$-Dirichlet problem \cite{HofPLe}. 

The key new ingredient in this work is the quantitative control we obtain on a Carleson measure $\nu$ built from the elliptic measure, defined by
\begin{equation}\label{nuintro.eq}
d\nu(x,r) = \frac{|\hm_L \ast (\nabla \varphi)_r(x)|^2}{|\hm_L \ast \varphi_r(x)|^2}  \, \frac{dx \, dr}{r},
\end{equation}
where $\varphi_r$ is a standard approximation of the identity. The control on $\nu$ is in the form of a point-wise density bound, which implies that $\nu$ is a Carleson measure, see \eqref{ptwisedensitybd1.eq} and Theorem \ref{mainquantest.thrm}. The kind of measure in \eqref{nuintro.eq} was introduced by Fefferman, Kenig and Pipher \cite{FKP} where they showed that a doubling weight $w$ is in the Muckenhoupt $A_\infty$ class if and only if $\nu_w$ is a Carleson measure, where $\nu_w$ is the measure formed by replacing $\hm_L$ by $w$ in \eqref{nuintro.eq}. Later, Korey \cite{Kor} investigated the case where the $\nu_w$ was a `vanishing' Carleson measure. Using this work of Korey we are able to show the following.
\begin{theorem}\label{mainvmo.thrm}
Let $L = -\div A\nabla$ be a divergence form uniformly elliptic operator on $\ree_+$, whose coefficient matrix $A$ satisfies the vanishing weak DKP condition (see Definition \ref{DKPconditions.def}) . If $k^\infty_L$ is the elliptic kernel associated to $L$ (in $\ree_+$) with pole at infinity then $\log k_L^\infty \in VMO(\rn)$. Moreover, if $k^X_L$ is the elliptic kernel associated to $L$ with pole at $X \in \ree_+$ then $\log k^X_L \in VMO_{loc}(\rn)$. Here $VMO$ is the space of vanishing mean oscillation and $VMO_{loc}$ is a local version of $VMO$ (see Definition \ref{BVMO.def}).
\end{theorem}
\begin{remark}
	By a simple change of variable argument, we can show the same conclusion holds when we replace $\ree_+$ by a $C^1$-square Dini domain. We say a domain $\Omega$ is $C^1$-square Dini if locally $\Omega$ is the region above the graph of a $C^1$ function $\varphi: \rn \to \RR$, which satisfies that
	\begin{equation}\label{eq:moc}
		|\nabla \varphi(x) - \nabla \varphi(y)| \leq \theta(|x-y|), \quad \text{ for any } x, y \in \RR^n, 
	\end{equation} 
	and
	\begin{equation}\label{def:squareDini}
		\int_0^* \theta(r)^2 \frac{dr}{r} < +\infty. 
	\end{equation} 
	We defer the proof of this remark to the Appendix.
\end{remark}

Theorem \ref{KPalt.thrm}, which constitutes a `large constant' version of Theorem \ref{mainvmo.thrm}, is also new. In particular, in \cite{DPP} the authors use a slightly stronger assumption on the $L^\infty$-oscillation of the coefficient matrix, and in \cite{KP} an even stronger condition on the gradient (see the DKP condition in Definition \eqref{DKPconditions.def}) is imposed. 
%We prove this `large constant' version as Theorem \ref{KPalt.thrm} for the pole at infinity. 
The proof of the `small constant' version of Theorem \ref{mainvmo.thrm} requires revisiting the work of Fefferman, Kenig and Pipher \cite{FKP}, see \cite{BES} and Section \ref{kpdpp.sect}.

The advantage of using the measure $\nu$ as in \eqref{nuintro.eq} is that it allows us to use the `Riesz formula' to shift the analysis from the elliptic measure (on the boundary) to the Green function (in the domain). For some time this approach seemed promising to the authors, but the necessary tools to complete the argument were lacking. 
The recent work of David, Li and Mayboroda \cite{DLM} provides the missing tools. In \cite{DLM}, the authors show that the gradient of the Green function is almost purely in the transversal direction in terms of a Carleson measure (see Theorem \ref{DLMeng.thrm}). They also prove a `Hardy inequality'-type lemma for weak DKP coefficients (see Lemma \ref{gammabdalpha.lem}). The estimates in \cite{DLM} in conjuction with the aforementioned point-wise density estimate \eqref{ptwisedensitybd1.eq} are used to prove Theorem \ref{mainquantest.thrm}. To the best of our knowledge this is the first time the measure defined in \eqref{nuintro.eq} has been used in this way. In \cite{FKP}, these Carleson measures were used to produce counterexamples (see Section 4 therein), not to prove that the elliptic measure was an $A_\infty$ weight.

To put our result in context it should be noted that the closely related work of \cite{DPP} places the elliptic kernel $k_L$ in the (local) reverse H\"older class $RH_p$ for all $p > 1$, when $L$ satisfies a condition slightly stronger than in Theorem \ref{mainvmo.thrm}. As $RH_p$ is a stronger condition for larger $p$ and $k_L \in RH_p$ implies $k_L\, dx \in A_\infty$ one might be led to believe that our work could be deduced from this fact, under this slightly stronger hypothesis. This is not the case, as there are weights $w = f \, dx$ that are in every reverse H\"older class that fail to have the property that $\log f \in VMO(\rn)$ (or $f \in VMO_{loc}(\rn)$). (An equivalent way to phrase this is $A_{\infty, as} \subsetneq \bigcap_{p>1} RH_p$, see the characterizations of `asymptotic $A_\infty$' in Theorem \ref{Koreyequiv.thrm} below.) As an example one can take $f(x) = 1_H(x) + (1+\epsilon)1_{H^c}(x)$ for any half space $H$ and $\epsilon > 0$. On the other hand, the role of the $RH_p$ condition is known, even in very rough settings, to be equivalent to the solvability of the $L^{p'}$ Dirichlet problem where $p' = p/(p-1)$ is the dual exponent (see \cite[Proposition 4.5]{HofPLe}). The condition that $\log k_L$ has small BMO norm implies that $k \in RH_p$ for large $p$ and hence the $L^p$-Dirichlet problem is solvable for a wider range of $p$; however, by the example above, the converse is not true. For this reason, $\log k_L \in VMO(\rn)$ has been considered an `asymptotically optimal' condition for the elliptic kernel. 

The condition $\log k_L \in VMO$ has appeared in many works and we give a few important examples here. For the Poisson kernel ($L = -\Delta$), the condition $\log k \in VMO$ was shown by Jerison and Kenig \cite{JK-VMO} for $C^1$ domains, and Kenig and Toro \cite{KT-Duke} under `vanishing flatness' condition for the domains. This is the natural endpoint to the work of Alt and Caffarelli, and Jerison \cite{AC, Jerison-Calpha}. For more general elliptic operators, it was shown by Escauriaza \cite{Esc} in Lipschitz domains and Milakis, Pipher and Toro \cite{MPT} in chord arc domains that the property that $\log k \in VMO$ is stable under `vanishing perturbations' of the coefficients, measured by Carleson measure as in \cite{FKP}. In \cite{BTZ} we
showed that $\log k_L \in VMO$ when $A$ is H\"older continuous and the corresponding operator $L$ is defined in a vanishing chord arc domain. We view these operators as a perturbations of constant coefficient ones. All of the results mentioned above where $\log k_L \in VMO$ are of a perturbative nature. This is part of what made Theorem \ref{mainvmo.thrm} somewhat elusive, as the conditions on the matrix $A$ make it difficult to view it as a `suitable' perturbation of a good operator at all scales. On the other hand, we point out David, Li and Mayboroda \cite{DLM} are able to obtain their estimates by using a perturbative regime at each scale, and we use their results in our work. 

The paper is organized as follows. In Section \ref{prelim.sect} we lay out the setting, notation and the analysis tools used throughout the paper. In Section \ref{PDEDLM.sect} we describe the classical PDE tools used throughout as well as the main result of \cite{DLM}. In Section \ref{quantest.sect} we prove Theorem \ref{mainquantest.thrm}, the foundation of our work. In Section \ref{mtproof.sect} we prove Theorems \ref{mainvmopt1.thrm} and \ref{mainvmopt2.thrm}, which combine to give Theorem \ref{mainvmo.thrm}. In Section \ref{kpdpp.sect} we contrast our work with \cite{KP} and \cite{DPP}.

{\bf Acknowledgement}. We would like to thank Linhan Li for some helpful comments on an earlier version of this work.

\section{Preliminaries and Notation}\label{prelim.sect}
Throughout $n \in \mathbb{N}$, $n \ge 2$ is a fixed constant. We work in $\ree_+:= \{(x,t) \in \rn \times \re: t > 0\}$. We use capital letters $X,Y, Z$ to denote points in $\ree$ and lowercase letters $x,y,z$ to refer to points in $\rn \times \{0\}$ (often identified with $\rn$). For two positive numbers, $a$ and $b$, we write $a \lesssim b$ whenever there exists a constant $C \ge 0$ such that $a/b < C$ such that $C$ depends only on the allowed structural constants in the statement of a definition, lemma, theorem, etc.. Similarly we write $a \approx b$ if there exists $C \ge 1$ such that $C^{-1} \le a/b \le C $.

The operators and matrices we work with satisfy an ellipticity condition.
\begin{definition}[Elliptic Matrices and Operators]
Fix $\Lambda \ge 1$. We say a matrix-valued function $A: \ree_+ \to M_{n+1}(\re)$ is $\Lambda$-elliptic if $\|A\|_{L^\infty(\ree)} \le \Lambda$ and
\[\langle A(X) \xi, \xi \rangle \ge \Lambda^{-1} |\xi|^2, \quad \forall \xi \in \ree, X \in \ree_+.\]
We say $A$ is elliptic if it is $\Lambda$-elliptic for some $\Lambda \ge 1$. The smallest constant $\Lambda \ge 1$ such that $A$ is $\Lambda$-elliptic is called the ellipticity constant of $A$.
We say $L$ is a divergence form elliptic operator (on $\ree_+$) if $L = -\div A \nabla$ (viewed in the weak sense) for an elliptic matrix $A$. In particular, $\Omega \subseteq \ree_+$ open we say $u \in W^{1,2}_{loc}(\om)$ is a weak solution to $Lu = 0$ in $\om$ if 
\[\iint A \nabla u \cdot \nabla F \, dX = 0, \quad \forall F \in C_c^\infty(\om).\]
\end{definition}

\begin{remark}
This definition of ellipticity is the most common in the literature; however we could just as well replace $A$ by $\widetilde{A} = A/\|A\|_{L^\infty(\ree)}$ and the theorems would only depend on the `lower ellipticity' of $\widetilde{A}$. This way there is no `artificial' dependence on ellipticity that is introduced when $A$ is multiplied by a constant (a function $u$ is a solution to $-\div A \nabla u = 0$ if and only if it is a solution to $-\div cA \nabla u = 0$).
\end{remark}

This work concerns a family of canonical measures associated to divergence form elliptic operators. These measures are known (together) as elliptic measure.

\begin{definition}[Elliptic measure and the Green function]\label{def:Green}
Let $L = -\div A \nabla$ be a divergence form elliptic operator on $\ree_+$. There exists a family of Borel measures on $\rn$, $\{\hm^X_L\}_{X \in \ree_+}$, such that for $f \in C_c^\infty(\rn)$ the function
\[u(X) = \int_{\rn} f(y) \, d\hm_L^X(y)\]
is the unique weak solution to the Dirichlet problem
\[(D)_L \begin{cases}
Lu = 0 \in \ree, \\
u|_{\rn} = f
\end{cases}\]
satisfying $u \in C(\overline{\ree_+} \cup \{\infty\})$. In particular $u(X) \to 0$ as $|X| \to \infty$ in $\ree_+$. We call the measure $\hm^X_L$ the elliptic measure with pole at $X$.

By \cite[Lemma 2.25]{HMT}, there is a Green function associated to $L$ in $\ree_+$, $G_L(X,Y): \ree_+ \times \ree_+ \setminus  \text{diag}(\ree_+) \to \re$, which satisfies the following. For fixed $X \in \ree$ the Green function can be extended, as a function in $Y$, to a function that vanishes continuously on the boundary $\rn$. The following `Riesz formula' holds and connects the elliptic measure and the Green function: If $f \in C_c^\infty(\rn)$ and $F \in C_c^\infty(\ree)$ are such that $F(y,0) = f(y)$ then
\begin{equation}\label{Rieszform.eq}
\int_{\rn} f(y) \, d\hm_L^X(y) - F(X) = - \iint A^T(y,s)\nabla_{y,s}G_L(X,(y,s)) \cdot \nabla_{y,s} F(y,s) \, dy \, ds.
\end{equation}
Here, and in the sequel, $A^T$ is the transpose of $A$. In our applications of \eqref{Rieszform.eq}, $F(X)$ will be equal to zero. We emphasize that \eqref{Rieszform.eq} also implies $u(\cdot) = G_L(X, \cdot)$ is a solution to $L^T u = 0$ in $\ree_+ \setminus \{X\}$ and we have remarked that $u$ vanishes continuously on $\rn$.
\end{definition}

In order to define the (weak) DKP condition we need some more notation. 
\begin{itemize}
\item We write $|E|$ for the Lebesgue measure of a set $E$. 
\item We define the integral averages $\fint_{E'} f \, dx = \tfrac{1}{|E'|} \int_{E'} f \, dx$ and $\fiint_{E} F \, dX = \tfrac{1}{|E|} \iint_{E}F \, dX$, for $E' \subset \rn$ and $E \subset \ree$ are sets of positive and finite measure (here $|\cdot|$ is the Lebesgue measure in the appropriate dimension).
\item For $x \in \rn$ and $r > 0$ we define $\Delta(x,r):= \{y \in \rn: |x- y| < r\}$, as usual we naturally identify $\Delta(x,r)$ as a subset of $\rn \times \{0\}$. When we make this identification we call $\Delta(x,r): = \{(y,0) \in \rn: |x- y| < r\}$ a surface ball.
\item Given $x \in \rn$ and $r > 0$ we define the Whitney region 
\[W(x,r) := \Delta(x,r) \times (r/2,r].\]
\item For $X \in \ree$ and $r > 0$ we let $B(X,r)$ denote the usual $n+1$ Euclidean ball.
\item For $\Delta = \Delta(x_0, r_0)$ or $B = B(x_1, r_1)$ we use the notation $r(\Delta) = r_0$ and $r(B) = r_1$ to denote the radius.
\item For $x \in \rn$ (identified with $\rn \times \{0\}$ and $r> 0$ we define the Carleson region
\[T(x,r) := B(x,r) \cap \ree_+.\]
\item If $\Delta = \Delta(x,r)$ we set $T_\Delta = T(x,r)$.
\item For $\Lambda \ge 1$, we let $\mathfrak{A}(\Lambda)$ denote the collection of all {\it constant} $\Lambda$-elliptic matrices.
\end{itemize}

\begin{definition}[Oscillation Coefficients]\label{OSCcoef.def}
Let $A$ be a $\Lambda$-elliptic matrix-value function on $\ree$. We define the following coefficients which measure the oscillation of $A$ on various regions. For $x \in \rn$ and $r > 0$ we define:
\[\alpha_2(x,r) = \inf_{A_0 \in \mathfrak{A}(\Lambda)}\left(\fiint_{(y,s) \in W(x,r)} |A(y,s) - A_0|^2 \right)^{1/2},\]
and
\[\gamma(x,r) = \inf_{A_0 \in \mathfrak{A}(\Lambda)}\left(\fiint_{(y,s) \in T(x,r)} |A(y,s) - A_0|^2 \right)^{1/2}.\]
If, in addition, $A$ is locally Lipschitz, we define for $x \in \rn$ and $r > 0$
\[\tilde{\alpha}(x,r) = r \sup_{(y,s) \in W(x,r)} |\nabla A(y,s)|.\]
It holds that $\alpha_2(x,r) \lesssim \tilde{\alpha}(x,r)$ and $\alpha_2(x,r) \le 2\gamma(x,r)$.
\end{definition}

We need one more definition before we introduce the class of coefficients we work with.
\begin{definition}[Carleson Measures]\label{def:Carl}
Let $\mu$ be a Borel measure on $\ree_+$. We say $\mu$ Carleson measure if
\[\|\mu\|_{\mathcal{C}}: = \sup_{\Delta} |\Delta|^{-1} \mu(T_\Delta) < \infty,\]
where the supremum is over all $n$-dimensional balls $\Delta$ in $\rn$.
Roughly speaking, this means that $\mu$ acts like an $n$-dimensional measure at the boundary. We call $\|\mu\|_{\mathcal{C}}$ the Carleson norm of $\mu$. We also define a localized Carleson norm. For $\Delta_0$ a surface ball and $\nu$ a Borel measure on $T_{\Delta_0}$, we define
\[\|\nu\|_{\mathcal{C}(\Delta_0)} : = \sup_{\Delta \subset \Delta_0} |\Delta|^{-1} \mu(T_\Delta)\]
and if $\|\nu\|_{\mathcal{C}(\Delta_0)}  < \infty$ we say $\nu$ is a Carleson measure on $T_{\Delta_0}$.

We say $\mu$ is a Carleson measure with vanishing trace (or simply vanishing Carleson measure) if $\mu$ is a Carleson measure and
\begin{equation}\label{def:vCarl}
	\lim_{r_0 \to 0^+} (\sup_{x_0 \in \rn} \|\mu\|_{\mathcal{C}(\Delta(x_0,r_0))}) = 0.
\end{equation} 
\end{definition}

\begin{definition}[DKP, weak DKP and vanishing weak DKP conditions]\label{DKPconditions.def}
Let $A$ be a $\Lambda$-elliptic matrix-valued function defined on $\ree_+$. 
\begin{itemize}
\item We say $A$ satisfies the DKP condition if $A$ is locally Lipschitz and $\mu$ defined by 
\[d\mu(x,r) = \tilde{\alpha}(x,r)^2 \, \frac{dx \, dr}{r}\]
is a Carleson measure.
\item We say $A$ satisfies the weak DKP condition if $\mu$ defined by 
\[d\mu(x,r) = \alpha_2(x,r)^2 \, \frac{dx \, dr}{r}\]
is a Carleson measure.
\item We say $A$ satisfies the vanishing weak DKP condition if $\mu$ defined by 
\[d\mu(x,r) = \alpha_2(x,r)^2 \, \frac{dx \, dr}{r}\]
is a Carleson measure with vanishing trace.
\item  We say $A$ satisfies the weak-DKP condition on $T_{\Delta_0}$ if $\mu$ defined by 
\[d\mu(x,r) = \alpha_2(x,r)^2 \, \frac{dx \, dr}{r}\]
is a Carleson measure on $T_{\Delta_0}$.
\end{itemize}
As observed above $\alpha_2(x,r) \lesssim \tilde{\alpha}(x,r)$ so that the DKP condition implies the weak DKP condition. %\footnote{In fact it was observed in \cite{DPP} that at  the level of the questions we are after, these two conditions are essentially equivalent (see also the proof of \cite[Corollary 10.3]{HMMTZ}). This is because there is a good perturbative theory, see \cite[Corollary 2.3]{DPP}. } 
For this reason, we will only work with the weak DKP condition in the sequel.
\end{definition}

\begin{remark}\label{AvsAt.rmk}
 The quantities $\alpha_2(x,r)$,  $\tilde{\alpha}(x,r)$ and $\gamma(x,r)$ do not see the difference between $A$ and $A^T$. In particular, any Carleson condition involving $\alpha_2(x,r)$,  $\tilde{\alpha}(x,r)$ and $\gamma(x,r)$ (like those in Definition \ref{DKPconditions.def}) holds for $A$ if and only if it holds for $A^T$. This will be particularly important as we intend to apply the estimates from \cite{DLM} to Green functions, $G(X, \cdot)$ which are solutions to $L^T$ away from $X$.
\end{remark}

Though $\alpha_2(x,r) \leq 2\gamma(x,r)$, $\gamma(x,r)$ is in general not bounded pointwise by $\alpha_2(x,r)$. However, their Carleson measures are essentially equivalent:
%
%We use two estimates from \cite{DLM}. We are in position to state the first one now.
\begin{lemma}[{\cite[Remark 4.22]{DLM}}]\label{gammabdalpha.lem}
Suppose that $A$ is a $\Lambda$-elliptic matrix-valued function defined on $\mathbb{R}^{n+1}_+$. Then
\[\left\|\gamma(x,r)^2 \, \frac{dx \, dr}{r} \right\|_{\mathcal{C}(\Delta_0)} \le C \left\|\alpha_2(x,r)^2 \, \frac{dx \, dr}{r} \right\|_{\mathcal{C}(3\Delta_0)},\]
and
\[\gamma(x,r)^2 \le C \left\|\alpha_2(x,r)^2 \, \frac{dx \, dr}{r} \right\|_{\mathcal{C}(3\Delta_0)} \forall (x,r) \in T_{\Delta_0},\]
where $C$ only depends on dimension.
\end{lemma}

Later, we will want to verify that elliptic kernels ($\tfrac{d\hm_L}{dx}$) exist and are in the function spaces $VMO$ or $VMO_{loc}$. We define these spaces now.
\begin{definition}[$BMO$, $VMO$ and $VMO_{loc}$]\label{BVMO.def}
Let $f \in L^1_{loc}(\rn)$. We say $f \in BMO(\rn)$ (or $f$ has bounded mean oscillation) if 
\[\|f\|_{BMO} := \sup_{r > 0} \sup_{x \in \rn} \fint_{\Delta(x,r)} \left|f(z) -\fint_{\Delta(x,r)}f(y)\,dy\right| \, dz < \infty.\]
We say $f \in VMO$ (or $f$ has vanishing mean oscillation) if $f$ is in BMO and 
\[ \lim_{r_0 \to 0^+} \sup_{r \in (0, r_0)} \sup_{x \in \rn} \fint_{\Delta(x,r)} \left|f(z) -\fint_{\Delta(x,r)}f(y)\,dy\right| \, dz = 0.\]
We say $f \in VMO_{loc}(\rn)$ if for every compact set $K \subset \rn$
\[ \lim_{r_0 \to 0^+} \sup_{r \in (0, r_0)} \sup_{x \in K} \fint_{\Delta(x,r)} \left|f(z) -\fint_{\Delta(x,r)}f(y)\,dy\right| \, dz = 0.\]
Notice the $VMO_{loc}(\rn)$ condition does not require $f$ to be in $BMO(\rn)$.
\end{definition}

We also want to investigate when $k_L$ is in the $A_\infty$ class (locally). Basic facts about $A_\infty$ weights can be found in \cite{GCRDF}.

\begin{definition}[$A_\infty$ weights and $A_\infty$ measures]\label{Ainfty.def}
We say a function $w$ is a weight if  $w \in L^1_{loc}(\rn)$ and  $w \ge 0$. 
A weight $w$ is said to be in $A_\infty$ if there exists $C$ such that
\begin{equation}\label{A_inftydef.eq}
\fint_{\Delta(x,r)} w(z) \, dz \le C\exp\left\{\fint_{\Delta(x,r)} \log w(z) \, dz\right\}, \quad \forall x \in \rn, r > 0.
\end{equation}
The infimum over constants $C$ such that the inequality \eqref{A_inftydef.eq} holds is called the $A_\infty$ constant, written $[w]_{A_\infty}$.

If $w \in A_\infty$ then there exists $p > 1$ and $C'$ both depending on dimension and $[w]_{A_\infty}$ such that $w$ satisfies the reverse H\"older inequality, with exponent $p$, that is,
\begin{equation}\label{rhpdef.eq}
\left(\fint_{\Delta(x,r)} w^p \, dz \right)^{1/p} \le C' \fint_{\Delta(x,r)} w \, dz, \quad \forall x \in \rn, r > 0.
\end{equation}
Conversely, any weight satisfying \eqref{rhpdef.eq} for some $p > 1$ is an $A_\infty$ weight with $[w]_{A_\infty}$ depending on $C'$ and $p$.

We say a Radon measure $\omega$ on $\mathbb{R}^n$ is in the $A_\infty$ class, if $\omega$ is absolutely continuous\footnote{Here we take the definition that $|E| = 0$ implies $\omega(E) =0$, so that absolute continuity is equivalent to the existence of a locally integrable density.} with respect to Lebesgue measure in $\mathbb{R}^n$ and its density $w := \tfrac{d\omega}{dx}$ is an $A_\infty$ weight.\footnote{In particular, an $A_\infty$ measure must be doubling. In fact as is well known in the theory of weights, a measure in the $A_\infty(dx)$ class satisfies the property that $\frac{|E|}{|\Delta|} \leq C \left(\frac{\omega(E)}{\omega(\Delta)} \right)^{\theta}$ for any surface ball $\Delta \subset \mathbb{R}^n$ and any set $E \subset \Delta$. Thus the doubling of $\omega$ simply follows from the doubling of the Lebesgue measure.}
\end{definition}

We use the following characterization of $A_\infty$, which is a modest improvement of a particular case of \cite[Theorem 3.1]{FKP}. Let $\varphi \in C^\infty_c(\Delta(0,1))$ be a radial function with $\varphi \equiv 1$ on $\Delta(0,1/2)$ and $0 \le \varphi \le 1$. Set $\psi := \nabla \varphi$ and use the notation $f_r(x) := r^{-n} f(x/r)$.

\begin{theorem}\label{FKPthrm.thrm}
%Let $\varphi \in C^\infty_c(\Delta(0,1))$ be a radial function with $\varphi \equiv 1$ on $\Delta(0,1/2)$ and $0 \le \varphi \le 1$. Set $\psi := \nabla \varphi$ and use the notation $f_r(x) := r^{-n} f(x/r)$. 
Let $\omega$ be a Radon measure on $\mathbb{R}^n$.
Then $\omega$ is in the $A_\infty$ class if and only if $\omega$ is doubling and the measure $\mu$ on $\ree_+$ defined by 
\begin{equation}\label{def:mu}
d\mu(x,r): = \frac{|\omega \ast \psi_r(x)|^2}{|\omega \ast \varphi_r(x)|^2} \, \frac{dx \, dr}{r}
\end{equation}
is a Carleson measure. 
Here we use the standard definition of convolution against a measure, that is,
\[\omega \ast f(x) := \int_{\rn} f(x-y) d\omega(y).\]

Moreover, the relationship between $[\omega]_{A_\infty}$ and  $ \|\mu\|_{\mathcal{C}}$ is quantitative in the sense that if $w = \tfrac{d\omega}{dx}$, then  $[w]_{A_\infty} \le F_1(n, C_{doub},\varphi,  \|\mu\|_{\mathcal{C}})$ and $ \|\mu\|_{\mathcal{C}} \le F_2(n, C_{doub}, \varphi, [w]_{A_\infty})$.
\end{theorem}
\begin{remark}	
	The above theorem is false when $\omega$ is not doubling, as is pointed out in \cite{FKP} by the example of the weight $w_k(x) := \min\{1/|x|, k \} $	for large values of $k$.
\end{remark}

The original statement in \cite[Theorem 3.1]{FKP} is for \underline{weights}, or more specifically, under the hypothesis that the Radon measure $\omega$ is equal to $w \, dx$ for a locally integrable function $w$. Our contribution here is the observation that the Carleson condition on $\mu$, in fact, implies the absolute continuity of $\omega$ with respect to Lebesgue measure. We give the proof of this fact below. This combined with \cite[Theorem 2.15]{FKP} finishes the proof of Theorem \ref{FKPthrm.thrm}.

\begin{lemma}\label{FKPabscty.thrm}
	Let $\omega$ be a Radon measure on $\mathbb{R}^n$ which satisfies the doubling property. Suppose that the measure $\mu$ defined in \eqref{def:mu} is a Carleson measure. Then $\omega$ is absolute continuous with respect with the Lebesgue measure.
\end{lemma}
\begin{proof}
%	Let $\omega$ be a Radon measure on $\mathbb{R}^n$ such that
%\begin{equation}\label{def:nu}
%	d\nu(x,t) := \frac{|\psi_t \ast \omega|^2}{|\varphi_t \ast \omega|^2} \frac{dx \, dt}{t} 
%\end{equation} 
%is a Carleson measure, where $\varphi$ is a function in the Schwartz space (i.e. a smooth function whose derivatives are rapidly decreasing at infinity) such that $\int \varphi \, dx = 1$ and $\psi = \nabla \varphi$. We will show that $\omega$ is absolutely continuous with respect to $dx$, the Lebesgue measure in $\mathbb{R}^n$.

By \cite{FKP}, in particular\footnote{Note that while \cite[Lemma 3.12]{FKP} is stated for (convolutions with) weights, the proof goes through without modification for (convolutions with) measures.} \cite[Lemma 3.12]{FKP} and its converse direction, we may replace $\varphi$ with the Gaussian kernel $\phi(x) = c_n e^{-|x|^2}$ and replace $\psi$ with $\nabla \phi$. Let $\{\epsilon_i\}$ be a sequence of positive numbers tending to zero, and we define a weight $\omega_i$ as follows
\begin{equation}\label{def:omegai}
	\omega_i(x): = \phi_{\sqrt{\epsilon_i}} \ast \omega(x) = \frac{1}{\epsilon_i^{n/2}} \int \phi\left( \frac{x-y}{\sqrt{\epsilon_i}} \right) d\omega(y). 
\end{equation} 
Clearly $\omega_i \rightharpoonup \omega$ as Radon measures, and each $\omega_i$ is a doubling measure with constant only depending on the doubling constant of $\omega$. 
We claim that $\omega_i \in A_\infty(dx)$ and the $A_\infty$ norms of the $\omega_i$'s are uniform in $i$, or equivalently, there exists a constant $C$ (independent of $i$) such that for any surface ball $\Delta \subset \mathbb{R}^n$,
\begin{equation}\label{eq:Ainfty}
	\left| \log \left( \fint_{\Delta} \omega_i \, dx \right) - \fint_{\Delta} \log \omega_i \, dx \right| \leq C. 
\end{equation} 
Notice that this inequality gives \eqref{A_inftydef.eq} with a constant $e^C$.
%\begin{equation}\label{def:nui}
%	d\nu_i(x,t) := \frac{|\psi_t \ast \omega_i|^2}{|\phi_t \ast \omega_i|^2} \frac{dx \, dt}{t} 
%\end{equation} 
%is a Carleson measure with a constant uniform in $i$.

We first consider small scales, i.e. when the radius of the surface ball satisfies $r_\Delta \leq \sqrt{\epsilon_i}$. Since $\omega$ is a doubling measure, by the definition \eqref{def:omegai} and the rapid decay of $\phi$ we have that
\[ \omega_i(x) \approx \frac{1}{\epsilon_i^{n/2}} \omega\left(\Delta(x, \sqrt{\epsilon_i}) \right). \]
Hence for any pair $x,y \in \mathbb{R}^n$ such that $|x-y| \leq \sqrt{\epsilon_i}$, we have that
\begin{align*}
	\omega_i(x) \approx \frac{1}{\epsilon_i^{n/2}} \omega\left(\Delta(x, \sqrt{\epsilon_i}) \right) \lesssim \frac{1}{\epsilon_i^{n/2}} \omega\left(\Delta(y, 2\sqrt{\epsilon_i}) \right) \lesssim \frac{1}{\epsilon_i^{n/2}} \omega(\Delta(y, \sqrt{\epsilon_i})) \approx \omega_i(y).
\end{align*}
Therefore \eqref{eq:Ainfty} follows with a constant only depending on the doubling constant of $\omega$.

Next we prove \eqref{eq:Ainfty} for large scales, i.e. when $r_\Delta \geq \sqrt{\epsilon_i}$.
Let $H_t$ denote the heat semigroup and $K(t, x, y)$ the heat kernel, i.e. for every $f\in C_c^\infty(\mathbb{R}^n)$ 
\[ H_t \circ f(x) = \int_{\mathbb{R}^n} K(t,x,y) f(y) \, dy = \frac{c_n}{t^{n/2}} \int_{\mathbb{R}^n}  e^{-\frac{|x-y|^2}{t}} f(y) \, dy. \] 
By definition
\begin{equation}\label{eq:heatkernelcov}
	H_t \circ f(x) = \phi_{\sqrt{t}} \ast f(x). 
\end{equation} 
and thus its spatial derivative satisfies
\begin{equation}\label{eq:gradheatkernel}
	\sqrt{t} \, \nabla \left( H_t \circ f \right)(x) = \sqrt{t} \, \nabla \left( \phi_{\sqrt{t}} \ast f \right)(x) =   \left( \nabla \phi \right)_{\sqrt{t}} \ast f(x). 
\end{equation} 
Moreover, \eqref{eq:heatkernelcov} and \eqref{eq:gradheatkernel} also hold when we replace $f$ by a doubling measure.
%In fact, when $\phi(x) = c_n e^{-|x|^2}$ simple computations show that
%\begin{align*}
%	\phi_t \ast \omega_i(x) = \phi_t \ast \left( \phi_{\epsilon_i} \ast \omega \right)(x) = \phi_{t+\epsilon_i} \ast \omega(x).
%\end{align*}
%We are implicitly using the semigroup property of the heat kernel, see the discussion in the footnote\footnote{For any $f\in L^2(\mathbb{R}^n)$, the function $u(x,t) := \phi_{\sqrt{t}} \ast f(x)$ satisfies the heat equation $\frac{\partial}{\partial t} = \frac14 \Delta u$ and $\lim_{t\to 0+} u(\cdot, t) = f$.}.
Denote 
\[ u(x,t) := H_t \circ \omega(x), \quad u_i(x,t):= H_t \circ \omega_i(x). \]
By the definition of $\omega_i$ and the semigroup property of the heat kernel, we have that
\begin{equation}\label{eq:trsltime}
	u_i(x,t) = H_t \circ \left( \phi_{\sqrt{\epsilon_i}} \ast \omega \right)(x) = H_t \circ \left( H_{\epsilon_i} \circ \omega \right)(x) = H_{t+\epsilon_i} \circ \omega(x) = u(x, t+\epsilon_i). 
\end{equation} 
Hence the spatial derivative satisfies
\begin{equation}\label{eq:gradtrsltime}
	 \nabla u_i(x,t) = \nabla u(x,t+\epsilon_i).
\end{equation}
By \eqref{eq:heatkernelcov}, \eqref{eq:gradheatkernel} and a change of variable ($t = \sqrt{r}$), it is easy to show that $d\mu$ defined in \eqref{def:mu} (with the Gaussian in place of $\varphi$) is a Carleson measure if and only if
\begin{equation}\label{eq:uCarl}
	\mathcal{C}(u) := \sup_{x_0 \in \mathbb{R}^n \atop{s>0} } \frac{1}{s^n} \int_0^{s^2} \int_{B_{s}(x_0)} \frac{|\nabla u(x,t)|^2}{u(x,t)^2} dx \, dt < +\infty,
\end{equation}
and moreover these two Carleson measure norms are equivalent.
%Analogously, in order to show that $d\nu_i$ in \eqref{def:nui} is a Carleson measure with a uniform constant, we need to estimate analogous integrals as in \eqref{eq:uCarl} for $u_i$. 
By \eqref{eq:trsltime} and \eqref{eq:gradtrsltime}, we have
\begin{align*}
	\int_0^{s^2} \int_{\Delta(x_0,s)} \frac{|\nabla u_i(x,t)|^2}{u_i(x,t)^2} dx\, dt & = \int_0^{s^2} \int_{\Delta(x_0,s)} \frac{|\nabla u(x,t + \epsilon_i)|^2}{u(x,t + \epsilon_i)^2} dx\, dt \\
	& = \int_{\epsilon_i}^{s^2 + \epsilon_i} \int_{\Delta(x_0,s)} \frac{|\nabla u(x,\tau)|^2}{u(x,\tau )^2} dx\, d\tau \\
	& \leq \int_{0}^{s^2 + \epsilon_i} \int_{\Delta(x_0,\sqrt{s^2+\epsilon_i})} \frac{|\nabla u(x,\tau)|^2}{u(x,\tau)^2} dx\, d\tau \\
	& \leq (s^2+\epsilon_i)^{n/2} \cdot \mathcal{C}(u).
\end{align*}
Therefore as long as $s\geq \sqrt{\epsilon_i}$, we have that
\begin{equation}\label{eq:uiCarl}
	\frac{1}{s^n} \int_0^{s^2} \int_{B_s(x_0)} \frac{|\nabla u_i(x,t)|^2}{u_i(x,t)^2} dx\, dt \leq \mathcal{C}(u) \frac{(s^2+\epsilon_i)^{n/2}}{s^n} \le 2^{n/2} \mathcal{C}(u). 
\end{equation} 
Applying the same argument as in the proof of \cite[Theorem 3.4]{FKP} to the weight $\omega_i$, we conclude that the Carleson-type estimate in \eqref{eq:uiCarl} implies the $A_\infty$-type estimate \eqref{eq:Ainfty}, also for large scales $r_\Delta \geq \sqrt{\epsilon_i}$. This finishes the proof of the claim.

 Since $\omega_i \in A_\infty(dx)$ with a constant independent of $i$, the following holds: For any $\epsilon>0$, there exists $\delta>0$ (depending only on $\epsilon$, not on $i$ or $R$) such that for every $R>0$,
\begin{equation}\label{eq:unifAinfty}
	\text{for any set } E\subset \Delta_{2R} \text{ satisfying } \frac{|E|}{|\Delta_{2R}|} < \delta, \text{ we have } \frac{\omega_i(E)}{\omega_i(\Delta_{2R})}< \epsilon, 
\end{equation}  
where we use $|\cdot|$ to denote the Lebesgue measure in $\mathbb{R}^n$ and the notation $\Delta_R := \Delta(0,R)$.

Fix $R > 0$ and $\epsilon > 0$, and let $E$ be an arbitrary set in $\Delta_R$ such that $|E|< \delta|\Delta_R|$, where $\delta$ is as above. By the outer approximation by open sets, there exists an open set $U \supset E$ such that $|U| < 2^n \delta|\Delta_R| = \delta|\Delta_{2R}|$. Without loss of generality we may assume that $U \subset \Delta_{2R}$. It follows from \eqref{eq:unifAinfty} that 
\[ \omega_i(U) < \epsilon \omega_i(\Delta_{2R}) \leq \epsilon \omega_i(\overline{\Delta_{2R}}) \] for every $i$. Since $\omega_i \rightharpoonup \omega$, $U$ is open and $\overline{\Delta_{2R}}$ is compact, we have
\[ \omega(E) \leq \omega(U) \leq \liminf_{i\to \infty} \omega_i(U) \leq \epsilon \limsup_{i\to \infty} \omega_i(\overline{\Delta_{2R}}) \le \epsilon \omega(\overline{\Delta_{2R}}) \leq C\epsilon \omega(\Delta_R), \]
where we use the doubling property of $\omega$ in the last inequality.
Thus, we have shown that for a fixed $R > 0$ and for every $\epsilon > 0$, there exists $\delta >0$ so that the following holds: 
\begin{equation}\label{eq:Ainfty}
	\text{Every } E \subset \Delta_R \text{ with } |E| < \delta|\Delta_R| \text{ satisfies } \omega(E) \leq C \epsilon \omega(\Delta_{R}).
\end{equation} 
In particular, this indicates that $\omega$ is absolute continuous with respect to the Lebesgue measure within the ball $\Delta_R$ for every $R>0$, and thus $\omega$ is absolutely continuous with respect to the Lebesgue measure on $\mathbb{R}^n$. In fact, a similar argument shows directly that $\omega$ is in the class $A_\infty(dx)$
\end{proof}

We also have the following characterizations for a weight to be
%It turns out that $\log w$ in $VMO$ if and only if a weight is 
\textit{`asymptotic  $A_\infty$'}. This was observed by Sarason \cite{Sarason} and thoroughly investigated by Korey \cite{Kor,Kor2}. 

\begin{theorem}[{\cite[Theorem 1]{Kor}, \cite[Theorem 10]{Kor2}}]\label{Koreyequiv.thrm}
Let $w$ be a weight.
The following are equivalent,
\begin{enumerate}
\item $w \in A_\infty$ and there exists $p > 1$ such that 
\[\lim_{r_0 \to 0^+} \sup_{x \in \rn} \sup_{r \in (0,r_0)} \frac{\left(\fint_{\Delta(x,r)} w^p \, dz \right)^{1/p}}{\fint_{\Delta(x,r)} w \, dz}  = 1.\]
\item $w \in A_\infty$ and for \underline{any} $p > 1$
\[\lim_{r_0 \to 0^+} \sup_{x \in \rn} \sup_{r \in (0,r_0)} \frac{\left(\fint_{\Delta(x,r)} w^p \, dz \right)^{1/p}}{\fint_{\Delta(x,r)} w \, dz}  = 1.\]
\item $w \in A_\infty$ and
\[\lim_{r_0 \to 0^+} \sup_{x \in \rn} \sup_{r \in (0,r_0)} \frac{\fint_{\Delta(x,r)} w \, dz}{\exp\left\{\fint_{\Delta(x,r)} \log w \, dz \right\}} = 1.\]

\item The measure $w \,dx$ is doubling ($\int_{\Delta(x,2r)} w \, dz \le C_{doub} \int_{\Delta(x,r)} w \, dz$) and measure $\mu$ on $\ree_+$ defined by 
\[d\mu(x,r): = \frac{|w \ast \psi_r(x)|^2}{|w \ast \phi_r(x)|^2} \, \frac{dx \, dr}{r}\]
as in Theorem \ref{FKPthrm.thrm} is a \underline{vanishing} Carleson measure.
\item $\log w \in VMO$.
\end{enumerate}
\end{theorem}

By inspection, \cite[Theorem 10]{Kor2} can be localized and we will use the following. 
\begin{theorem}[{\cite[Theorem 10]{Kor2}}]\label{VMOlocKorey.thrm}
Let $w$ be a weight. Then $\log w \in VMO_{loc}$ if and only if for every $R > 0$
\[\lim_{r_0 \to 0^+} \sup_{x \in \Delta(0,R)} \sup_{r \in (0,r_0)} \frac{\left(\fint_{\Delta(x,r)} w^2 \, dz \right)^{1/2}}{\fint_{\Delta(x,r)} w \, dz}  = 1.\]
\end{theorem}

\section{Classical Estimates and the \cite{DLM} energy estimates}\label{PDEDLM.sect}
We begin this section by recalling some classical estimates for positive solutions to divergence form elliptic equations in the upper half space that vanish at the boundary. After doing so, we will specialize to the case of operators whose coefficient matrix satisfies a (local) weak DKP condition and introduce the energy estimates proved in \cite{DLM}. The following lemma is explicitly stated and proved in \cite[Lemma 2.8]{DLM}.
\begin{lemma}\label{engcs.lem}
Let $L = -\div A \nabla$ be a divergence form elliptic operator (on $\ree_+$). Suppose that $x \in \rn$, $r > 0$ and $u \in W^{1,2}(T(x,2r))$  is a non-negative weak solution to $Lu = 0$ in $T(x,2r)$ which vanishes continuously on $\Delta(x,2r)$. Then there exist implicit constants, depending only on $n$ and the ellipticity constant of $A$ such that 
\[\fiint_{T(x,r)} |\nabla u(Y)|^2 \, dY \approx \frac{u(x,r)^2}{r^2}.\]
\end{lemma}

The following is a well-known estimate, often called the CFMS estimate.
\begin{lemma}[{\cite{CFMS}}]\label{CFMS.lem}
Let $L$ be a divergence form elliptic operator (on $\ree_+$). If $\hm^{X_0}_L$ is the elliptic measure for $L$ with pole $X_0 \in \ree_+$ and $G_L(X,Y)$ is the Green function for $L$ then
\[\frac{\hm^{X_0}_L(\Delta(x,r))}{|\Delta(x,r)|} \approx \frac{G_L(X_0, (x,r))}{r}\]
for $x \in \rn$ and $r > 0$, provided $X_0 \not\in T(x,2r)$. Here the implicit constants depend on $n$ and the ellipticity constant of $A$.
\end{lemma}

Combining the previous two lemmas we obtain the following.

\begin{lemma}\label{CFMSeng.lem}
Let $L$ be a divergence form elliptic operator (on $\ree_+$). If $\hm^{X_0}_L$ is the elliptic measure for $L$ with pole $X_0 \in \ree_+$ and $G_L(X,Y)$ is the Green function\footnote{Recall $G_L(X_0, \cdot)$ satisfies \eqref{Rieszform.eq} and $L^T G_L(X, \cdot) = \delta_X$.} for $L$  then
\[\frac{\hm^{X_0}_L(\Delta(x,r))}{|\Delta(x,r)|} \approx \left(\fiint_{T(x,r)} |\nabla_Y G_L(X_0, Y)|^2\, dY \right)^{1/2}\]
for $x \in \rn$ and $r > 0$, provided $X_0 \not\in T(x,4r)$. Here the implicit constants depend on $n$ and the ellipticity constant of $A$. More generally 
\[\frac{\hm^{X_0}_L(\Delta(x,r))}{|\Delta(x,r)|} \approx_M \left(\fiint_{T(x,Mr)} |\nabla_Y G_L(X_0, Y)|^2\, dY \right)^{1/2}\]
for $x \in \rn$ and $r > 0$, provided $X_0 \not\in T(x,4Mr)$, the implicit constants depend on $M$, $n$ and the ellipticity constant of $A$.
\end{lemma}
In the previous lemma, the second estimate follows from the first, Lemma \ref{engcs.lem} and the Harnack inequality (applied to $G(X_0,\cdot)$ a solution to $L^T u = 0$ away from $X_0$).  We also have the following doubling property for harmonic measure, which can be deduced from Lemma \ref{CFMS.lem} and the Harnack inequality. 

\begin{lemma}\label{hmdoub.lem}
Let $L$ be a divergence form elliptic operator (on $\ree_+$). If $\hm^{X_0}_L$ is the elliptic measure for $L$ with pole $X_0 \in \ree_+$ then
\[\hm_L^{X_0}(\Delta(x,2r)) \lesssim \hm^{X_0}_L(\Delta(x,r)) \]
provided that $X_0 \not\in T(x,4r)$. Here the implicit constants depend on $n$ and the ellipticity constant of $A$.
\end{lemma}
Notice the previous lemma does not give {\it global} doubling of the measure. The lemma only gives local doubling, up to the scale of the distance from the pole to the boundary. Later we would like to work with ``the" Green function and elliptic measure with pole at infinity and we introduce them with the following lemma, which can be proved just as in \cite[Corollary 3.2]{KT-ann}.

\begin{lemma}[Green function and elliptic measure at infinity]\label{def:Greeninf}
Let $L$ be a divergence form elliptic operator (on $\ree_+$) with Green function $G_L(X,Y)$. Define the sequence of functions $u_k: \overline{\ree_+} \to \re$ by 
\[u_k(Y) := \frac{G_L((0,2^{k}), Y)}{G_L((0,2^{k}), (0,1))},\]
where we have extended $u_k$ to the boundary by zero ($u_k(y,0) = 0, \ \forall y \in \rn$). There exists a subsequence $u_{k_j}$ such that $u_{k_j}$ converges uniformly on compact subsets of $\overline{\ree_+}$ to a function $U$ with the following properties.
\begin{itemize}
\item $U(y,0) = 0$ for all $y \in \rn$.
\item $U(0,1) = 1$.
\item $U(Y) > 0$ for all $Y \in \ree_+$.
\item $U \in C(\overline{\ree_+})$.
\item $U$ solves $L^T U = 0$ in $\ree_+$.
\end{itemize}
Moreover, there exists a locally finite measure $\hm^\infty_L$ on $\rn$ with
\[\frac{1}{G((0,2^{k_j}),(0,1))} \hm^{(0,2^{k_j})}  \rightharpoonup   \hm^\infty_L \]
such that the following Riesz formula holds
\begin{equation}\label{Rieszformulainfty.eq}
\int_{\rn} f(y) \, d\hm^\infty_L(y) = - \iint A^T(y,s)\nabla_{y,s}U(y,s) \cdot \nabla_{y,s} F(y,s) \, dy \, ds,
\end{equation}
whenever $f \in C_c^\infty(\rn)$ and $F \in C_c^\infty(\ree)$ are such that $F(y,0) = f(y)$
We call $\hm^\infty_L$ the {\bf elliptic measure with pole at infinity} and $U$ the {\bf Green function with pole at infinity}.
\end{lemma}

The estimates for $G$ in Lemma \ref{CFMSeng.lem} hold for $U$ (globally) and the measure $\hm^\infty_L$ is {\it globally} doubling. We summarize these facts in the following Lemma. 
\begin{lemma}\label{gfinftyfacts.lemma}
Let $L$ be a divergence form elliptic operator (on $\ree_+$). Let $U$ be the Green function with pole at infinity and $\hm^\infty_L$ is the elliptic measure with pole at infinity then the following hold:
\begin{itemize}
\item For $x \in \rn$ and $r > 0$
\[\frac{\hm^\infty_L (\Delta(x,r))}{|\Delta(x,r)|} \approx_M  \left(\fiint_{T(x,Mr)} |\nabla U(Y)|^2\, dY \right)^{1/2},\]
where the implicit constants depend on $M$, $n$ and the ellipticity constant of $A$.
\item For $x \in \rn$ and $r > 0$ it holds $\hm^\infty_L(\Delta(x,2r))\lesssim \hm^\infty_L(\Delta(x,r))$, where the implicit constant depends on $n$ and the ellipticity of $A$.
\end{itemize}
\end{lemma}

Next, we need to define some local energies as in \cite{DLM} (we define some additional objects as well).

\begin{definition}[Local energies]\label{localenergies.def}
Let $x \in \mathbb{R}^n$ and $r > 0$. Suppose $u \in W^{1,2}(T(x,r))$ and $\iint_{T(x,r)} |\nabla u|^2 \, dz \, dt$ is non-zero. Define
\[E_u(x,r) := \fiint_{T(x,r)} |\nabla u|^2 \, dz \, dt,\] 
and for $i = 1, \dots, n$
\[E_{u,i} (x,r) := \fiint_{T(x,r)} |\partial_{x_i} u|^2 \, dz \, dt.\]
Let
\[\lambda(x,r) := \fiint_{T(x,r)} \partial_{t} u \, dz \, dt. \]
We define, as in \cite{DLM},
\[J_u(x,r) = \fiint_{T(x,r)} |(\nabla_{x,t} u) - \lambda(x,r) e_{n+1}|^2 \, dz, dt,\]
which essentially measures how far $\nabla_{x,t}u $ is from its vertical component,
and define
\[\beta_u(x,r):= \frac{J_u(x,r)}{E_u(x,r)}.\]
A simple computation shows that 
\begin{equation}\label{betais.eq}
\beta_{u,i}(x,r):= \frac{E_{u,i}(x,r)}{E_u(x,r)} \le \beta_u(x,r).
\end{equation}
\end{definition}

%The second estimate we need from \cite{DLM} is the following.
We will need the following main theorem proven in \cite{DLM}.

\begin{theorem}[{\cite[Theorem 1.15]{DLM}}]\label{DLMeng.thrm}
Let $x_0 \in \mathbb{R}^n$ and $R > 0$ and $A$ be a $\Lambda$-elliptic matrix satisfying the weak DKP condition on $T(x_0,R)$. Suppose $u$ is a positive solution to $L u := -\div A \nabla u = 0$ in $T(x_0, R)$ and $u$ vanishes on $\Delta(x_0,R)$. Then for any $\tau \in (0,1/20]$ it holds
\[\left\|\beta_u(x,r) \, \frac{dx\, dr}{r}\right\|_{\mathcal{C}(\Delta(x_0,\tau R))} \le C\left(\tau^\eta  + \left\|\alpha_2(x,r)^2 \, \frac{dx\, dr}{r}\right\|_{\mathcal{C}(\Delta(x_0, R))}\right),\]
where the constants $C$ and $\eta$ depend only on $\Lambda$ and $n$.
\end{theorem}

As was observed above in \eqref{betais.eq}, $\beta_{u,i}(x,r) \le \beta_u(x,r)$, giving the following corollary.

\begin{corollary}\label{xengy.cor}
Let $x_0 \in \mathbb{R}^n$ and $R > 0$ and $A$ be a $\Lambda$-elliptic matrix satisfying the weak DKP condition on $T(x_0,R)$. Suppose $u$ is a positive solution to $L u := -\div A \nabla u = 0$ in $T(x_0, R)$ and $u$ vanishes on $\Delta(x_0,R)$. Then for any $\tau \in (0,1/20]$ it holds
\[\left\|\beta_{u,i}(x,r) \, \frac{dx\, dr}{r}\right\|_{\mathcal{C}(\Delta(x_0,\tau R))} \le C\left(\tau^\eta  + \left\|\alpha_2(x,r)^2 \, \frac{dx\, dr}{r}\right\|_{\mathcal{C}(\Delta(x_0, R))}\right),\]
where the constants $C$ and $\eta$ depend only on $\Lambda$ and $n$.
\end{corollary}

\section{A local quantitative estimate on the FKP Carleson measure for $\hm_L$}\label{quantest.sect}
In this section we prove Theorem \ref{mainquantest.thrm}, from which we will derive all of our other results. Below $\varphi \in C_c(B(0,1))$ is a fixed radial function with $\varphi(x) = 1$ for $|x| \le 1/2$, $0 \le \varphi \le 1$. For a function $f$  and $r > 0$ we use the notation $f_r(y) := r^{-n}f(y/r)$.

\begin{theorem}\label{mainquantest.thrm}
Suppose $x_0 \in \mathbb{R}^n$ and $R > 0$. Let $A$ be a $\Lambda$-elliptic matrix satisfying the weak DKP condition on $T(x_0, 100 R)$. Let $Z_0 = (z_0,t_0)\in  \ree_+ $ be any point with $t_0 > 100R$ and  $\omega = \omega^{X_0}_L$ be the elliptic measure associated to $L = -\div A \nabla$ on $\mathbb{R}^{n+1}_+$ with pole at $X_0$. 
Then the measure
\[d\nu(x,r) = \frac{|\omega \ast (\nabla \varphi)_r(x)|^2}{|\omega \ast \varphi_r(x)|^2} \frac{dx \, dr}{r}\]
is a Carleson measure in $\Delta(x_0,R)$. Moreover, there exists $\tau_0  \in (0,1/20]$ depending on dimension and $\eta = \eta(n, \Lambda)$ such that for all $\tau \in (0,\tau_0)$
\begin{equation}\label{mainquantest.eq}
\|\nu\|_{\mathcal{C}(\Delta(x_0,\tau R))} \le C\left(\tau^\eta +  \left\|\alpha_2(x,r)^2 \, \frac{dx\, dr}{r}\right\|_{\mathcal{C}(\Delta(x_0, 100 R))}\right),
\end{equation}
where $C = C(n, \Lambda)$.
If $A$ satisfies the (global) weak DKP condition, then the estimate \eqref{mainquantest.eq} as well as 
\begin{equation}\label{mainquantest2.eq}
\|\nu\|_{\mathcal{C}} \le C\left\|\alpha_2(x,r)^2 \, \frac{dx\, dr}{r}\right\|_{\mathcal{C}},
\end{equation}
%where $C = C(n, \Lambda)$ 
holds for
\[d\nu(x,r) = \frac{|\omega_L^\infty \ast (\nabla \varphi)_r(x)|^2}{|\omega_L^\infty \ast \varphi_r(x)|^2} \frac{dx \, dr}{r},\]
where $\omega_L^\infty$ is the elliptic measure with pole at infinity 
\end{theorem}

%\begin{remark}
	%In the proof of this theorem, we do not assume/need the absolute continuity of the elliptic measure $\omega = \omega_L^{X_0}$ (or $\omega_L^\infty$) with respect to the Lebesgue measure $\mathcal{L}^n$, and $\omega$ could just be a general doubling Radon measure, not necessarily a weight. In this case, its convolution with a function $f$ (real-valued or vector-valued) is defined as
%	\begin{equation}\label{convolutiondef.eq}
%		\omega \ast f(x) := \int_{\rn} f(x-y) d\omega(y). 
%	\end{equation}
%	When $\omega$ is indeed a weight the convolution defined above is exactly the same as the convolution $k \ast f(x)$, where $k = \frac{d\omega}{dx}$ is the elliptic kernel, and thus we do not distinguish the two.
%\end{remark}

\begin{proof}
We prove \eqref{mainquantest.eq} in the finite pole case. To show \eqref{mainquantest.eq} in the infinite pole case one can replace the use of both Lemma \ref{CFMSeng.lem} and Lemma \ref{hmdoub.lem} with Lemma \ref{gfinftyfacts.lemma}.
Let $x \in \rn$ and $r >0$ be such that $\Delta(x,r) \subset \Delta(x_0,R)$. Set $G(Y) = G(Z_0, Y)$ the Green function for operator $L$ with pole at $X_0$ and $\hm:= \hm_L^{Z_0}$. We note that $G(Y)$ solves $-\div A^T \nabla G = \delta_{Z_0}$, so that $G$ is a solution to $L^Tu = 0$ away from $Z_0$ vanishing on the boundary. In particular, since the matrix $A^T$ has the same $\alpha_2$ and $\gamma$ numbers as $A$ (see Remark \ref{AvsAt.rmk}), we may apply Corollary \ref{xengy.cor} and Lemma \ref{gammabdalpha.lem} to $u(Y) = G(Y)$ and we will do this later.

We will estimate the density
\[\frac{|\omega \ast (\nabla \varphi)_r(x)|^2}{|\omega \ast \varphi_r(x)|^2}.\]
We start with replacing the denominator by the energy $E_G(x,r)$. By the local doubling property of $\omega$ (Lemma \ref{hmdoub.lem}) and the properties of $\varphi$ we have
\[\omega \ast \varphi_r(x) \approx \frac{\omega(\Delta(x,r))}{|\Delta(x,r)|}.\]
Then using Lemma \ref{CFMSeng.lem} (with $L^T$ in place of $L$) we have
\[\omega \ast \varphi_r(x) \approx \left(\fiint_{T(x,2r)} |\nabla_Y G(Y)|^2 \ dY \right)^{1/2} = E_G(x,2r)^{1/2}.\]
Thus,
\begin{equation}\label{engreplaced.eq}
\begin{split}
\frac{|\omega \ast (\nabla \varphi)_r(x)|^2}{|\omega \ast \varphi_r(x)|^2} &\approx |\omega \ast (\nabla \varphi)_r(x)|^2 E_G(x,2r)^{-1} 
\\& \approx \sum_{i =1}^n |\omega \ast (\partial_{x_i} \varphi)_r(x)|^2 E_G(x,2r)^{-1}.
\end{split}
\end{equation}
The following claim will essentially prove the theorem.
\begin{claim}\label{mtmc.cl}
\[|\omega \ast (\partial_{x_i} \varphi)_r(x)|^2 \lesssim  \gamma(x, 2 r)^2 E_G(x,2 r) + \sum_{\ell = 1}^n E_{G,\ell}(x,2 r).\]
\end{claim}
\begin{proof}[Proof of Claim \ref{mtmc.cl}]
We need to make a few simple observations. We have $(\partial_{x_i} \varphi)_r(x) = r \partial_{x_i} \varphi_r(x)$. Let $g \in C_0^\infty([-1,1])$, $0 \le g \le 1$, $g = 1$ on $[-1/2,1/2]$ then define
\[\Phi_{x,r}(y,s) := \varphi_r (x -y) g(s/r).\]
Then $\partial_{y_i} \Phi_{x,r}(y,s)$ is a smooth extension of $\partial_{y_i} \varphi_r(x-y)$, that is, $\partial_{y_i} \Phi_{x,r}  \in C^\infty_0(B(x,2r))$ and $\partial_{y_i} \Phi_{x,r}(\cdot,0) = \partial_{y_i} \varphi_r(x-y)$. Moreover, for $j = 1,\dots n$, we have
\begin{equation}\label{Phibds.eq}
|\partial_t \partial_{y_i} \Phi_{x,r}(y,s)|,  |\partial_{y_j} \partial_{y_i} \Phi_{x,r}(y,s)| \lesssim \frac{1}{r^{n+2}}\chi_{B(x,2r)}(y,s),
\end{equation}
where the implicit constants depend on dimension alone.
The Riesz formula \eqref{Rieszform.eq} gives
\begin{align*}
\omega \ast (\partial_{x_i} \varphi)_r(x) &= r \int_{\mathbb{R}^n} \partial_{x_i} \varphi_r(x - y) \, d\omega(y) 
\\& =  r\iint_{\mathbb{R}^{n+1}_+} A^{T}(Y) \nabla_{y,s} G(y,s) \cdot \nabla_{y,s} \partial_{y_i} \Phi_{x,r}(y,s) \, dy \, ds,
\end{align*}
where we used that $\partial_{y_i} \Phi_{x,r}(Z_0) = 0$ and we are using the convention that $G = 0$ in the lower half space $\ree_-$.
Now we let $A_0 = A_{x,r}$ be a constant matrix attaining the infimum in the definition of $\gamma(x,r)$. We write 
\begin{align*}
& \iint_{\mathbb{R}^{n+1}_+} A^{T}(Y) \nabla_{y,s} G(y,s) \cdot \nabla_{y,s} \partial_{y_i} \Phi_{x,r}(y,s) \, dy \, ds 
\\& = \iint_{\mathbb{R}^{n+1}_+} (A^{T}(Y) - A^T_0)\nabla_{y,s} G(y,s) \cdot \nabla_{y,s} \partial_{y_i} \Phi_{x,r}(y,s) \, dy \, ds  
\\ & \quad + \iint_{\mathbb{R}^{n+1}_+} A_0^{T} \nabla_{y,s} G(y,s) \cdot \nabla_{y,s} \partial_{y_i} \Phi_{x,r}(y,s) \, dy \, ds.
\end{align*}
In summary, we have shown
\begin{equation}\label{numeratorbd.eq}
|\omega \ast (\partial_{x_i} \varphi)_r(x)| \le I + II,
\end{equation}
where
\[I := r\left| \iint_{\mathbb{R}^{n+1}_+} (A^{T}(Y) - A^T_0)\nabla_{y,s} G(y,s) \cdot \nabla_{y,s} \partial_{y_i} \Phi_{x,r}(y,s) \, dy \, ds\right|\]
and 
\[II: = r\left| \iint_{\mathbb{R}^{n+1}_+} A_0^{T} \nabla_{y,s} G(y,s) \cdot \nabla_{y,s} \partial_{y_i} \Phi_{x,r}(y,s) \, dy \, ds \right|.\]
First we handle $I$. By the Cauchy-Schwarz inequality and \eqref{Phibds.eq}
\begin{align}\label{1bound.eq}
I &\lesssim  \left(\fiint_{T(x,2r)} |A - A_0|^2 \, dY\right)^{1/2} \left(\fiint_{T(x,2r)} |\nabla_YG(Y)|^2 \, dY\right)^{1/2} 
\\ \nonumber & \le \gamma(x, 2 r) E_G(x,2 r)^{1/2}.
\end{align}
To handle term $II$ we write $\nabla_{y,s} G(y,s) = (\nabla_{y} G(y,s),0)^T+ (0,\partial_s G(y,s))^T$ to see
\begin{align*}
II & \le r\left| \iint_{\mathbb{R}^{n+1}_+} A_0^{T} (\nabla_y G(y,s),0)^T \cdot \nabla_{y,s} \partial_{y_i} \Phi_{x,r}(y,s) \, dy \, ds \right|
\\ & \quad + r\left| \iint_{\mathbb{R}^{n+1}_+} A_0^{T} (0, \partial_s G(y,s))^T \cdot \nabla_{y,s} \partial_{y_i} \Phi_{x,r}(y,s) \, dy \, ds \right|
\\ & = II_1 + II_2.
\end{align*}
To handle $II_1$ we use the boundedness of $A$ and \eqref{Phibds.eq} to see
\[II_1 \lesssim \fiint_{T(x,2r)} |\nabla_{y} G(y,s) | \, dy \, ds \lesssim \sum_{\ell = 1}^n E_{G,\ell}(x,2r)^{1/2}.\]
To handle $II_2$ we integrate by parts in $s$ noting that $G = 0$ on $\mathbb{R}^{n} \times \{0\}$ and $\Phi \in C_0^\infty(B(x,2r))$, but first we write out the matrix multiplication using the notation $(A_0)_{i,j} =: a^0_{i,j}$. We obtain 
\begin{align*}
II_2 &=r \left| \sum_{j = 1}^{n+1} \iint a^0_{ n+1,j} \partial_s G(y,s) (\partial_{y_j}\partial_{y_i} \Phi_{x,r}(y,s)) \, dy \, ds\right| 
\\& = r \left| \sum_{j = 1}^{n+1} \iint a^0_{ n+1,j}  G(y,s)  (\partial_s \partial_{y_j}\partial_{y_i} \Phi_{x,r}(y,s)) \, dy \, ds\right|,
\end{align*}
where we used the notation $\partial_{y_{n+1}} = \partial_s$ and integrated by parts in $s$. Now we integrate by parts in the $y_i$ variable using that $\Phi \in C_0^\infty(B(x,2r))$ and that $i \in \{1,\dots, n\}$ to obtain
\begin{equation}\label{22almfin.eq}
II_2 = r \left| \sum_{j = 1}^{n+1} \iint a_{j, n+1}^0  \partial_{y_i} G(y,s) (\partial_s \partial_{y_j} \Phi_{x,r}(y,s)) \, dy \, ds\right|.
\end{equation}
Now using the boundedness of $A$, \eqref{22almfin.eq} and \eqref{Phibds.eq} we obtain
\[II_2 \lesssim E_{G,i}(x,2r)^{1/2}.\]
Putting together our estimates for $II_1$ and $II_2$ we have
\[II \lesssim  \sum_{\ell = 1}^n E_{G,\ell}(x,2r)^{1/2}.\]
Combining this bound with \eqref{1bound.eq} and plugging into  \eqref{numeratorbd.eq} gives
\[ |\omega \ast (\partial_{x_i} \varphi)_r(x)| \lesssim \gamma(x, 2 r) E_G(x,2 r)^{1/2} + \sum_{i = 1}^n E_{G,i}(x,2r)^{1/2},\]
which proves the claim.
\end{proof}
Now using Claim \ref{mtmc.cl} and \eqref{engreplaced.eq}, we have
\begin{equation}\label{alfinqtest.eq}
\begin{split}
\frac{|\omega \ast (\nabla \varphi)_r(x)|^2}{|\omega \ast \varphi_r(x)|^2} &\lesssim \gamma(x, 2 r)^2 + \sum_{i = 1}^n \frac{E_{G,i}(x,2r)}{E_G(x,2r)}
\\ & \lesssim \gamma(x, 2 r)^2 + \sum_{i = 1}^n \beta_i(x,2r).
\end{split}
\end{equation}
Returning the the measure $\nu$ we have
\begin{equation}\label{ptwisedensitybd1.eq}
d\nu(x,r) \lesssim \left( \gamma(x, 2 r)^2 + \sum_{i = 1}^n \beta_i(x,2r) \right) \frac{dx \, dr}{r}.
\end{equation}
The estimate \eqref{mainquantest.eq} then follows by Corollary \ref{xengy.cor} and Lemma \ref{gammabdalpha.lem} (applied with $u = G$ and the operator $L^T = -\div A^T \nabla$, see Remark \ref{AvsAt.rmk}). To obtain \eqref{mainquantest2.eq}, we note that \eqref{mainquantest.eq} holds for $x_0 = 0$ and any $R > 0$ since the pole is infinite. Then for any fixed $\tau \in (0,\tau_n]$ it holds
\[\|\nu\|_{\mathcal{C}} = \sup_{R > 0} \|\nu\|_{\mathcal{C}(\Delta(0,\tau R))} \le C \left(\tau^\eta + \left\|\alpha_2(x,r)^2 \, \frac{dx\, dr}{r}\right\|_{\mathcal{C}}\right).\]
As $\tau$ can be taken arbitrarily small, this shows \eqref{mainquantest2.eq}.

\end{proof}

\section{Proof of Theorem \ref{mainvmo.thrm}}\label{mtproof.sect}
In this section, we prove Theorem \ref{mainvmo.thrm}. We will first `prove' the infinite pole case, which is immediate from Theorem \ref{mainquantest.thrm}, Lemma \ref{FKPabscty.thrm} and the work of Korey \cite{Kor}. To prove the finite pole case, we will need to introduce a change of pole argument and the `kernel function'. 

%{\color{red} Before continuing we need the following {\it qualitative} result from \cite{KP, DPP}.
%\simon{this red text will be removed and the footnote on the next page will also go away}
%\begin{theorem}[{\cite{KP,DPP}}]\label{qualKP.thrm}
%Let $L = -\div A\nabla$ be a divergence form elliptic operator on $\ree_+$, whose coefficient matrix $A$ satisfies the weak DKP condition. Then $\hm^{X_0}_L \ll \mathcal{L}^n$ for any $X_0 \in \ree_+$, where $\hm^{X_0}_L$ is the elliptic measure for $L$ with pole at $X_0$, and $\mathcal{L}^n$ is the Lebesgue measure on $\rn$. Moreover, $\hm^{\infty}_L \ll \mathcal{L}^n$, where $\hm^{\infty}_L$ is the elliptic measure for $L$ with pole at $\infty$.
%\end{theorem}
%Theorem \ref{qualKP.thrm} for the finite pole case was proven in \cite{KP,DPP}.\footnote{{\color{red}The authors in \cite{KP} first prove the theorem assuming the coefficient matrix satisfies the DKP condition. For coefficient matrix satisfying the {\it weak} DKP condition, one can show the same conclusion holds with the help of a perturbative result, see the footnote after Definition \ref{DKPconditions.def}.}} Even though the theorem in \cite{KP} was for bounded Lipschitz domains, we can use a localization argument to get the desired result in the upper half-space, see Section \ref{kpsecondlook.sect} and, in particular, the discussions in Remark \ref{KPlocalizationarg.rmk}. 
%The infinite pole case follows from a change of pole argument, see \eqref{compinfkern.eq} in Lemma \ref{infkernfnlem.lem} below.}

\begin{theorem}\label{mainvmopt1.thrm}
Let $L = -\div A\nabla$ be a divergence form elliptic operator on $\ree_+$, whose coefficient matrix $A$ satisfies the vanishing weak DKP condition. Let $\hm^\infty_L$ be the elliptic measure with pole at infinity. Then $\hm^\infty_L \ll \mathcal{L}^n$, $\hm^\infty_L  \in A_\infty$ and $k^\infty_L(y) := \tfrac{d\hm^\infty_L}{dx}(y)$ has the property that $\log k^\infty_L \in VMO(\rn)$. 
\end{theorem}
The function $k^\infty_L(y)$ is often referred to as the elliptic kernel with pole at infinity. 
\begin{proof}[Proof of Theorem \ref{mainvmopt1.thrm}]

By Theorem \ref{mainquantest.thrm} and the fact that $A$ satisfies the weak DKP condition it holds that the measure $\nu$ defined by
\[d\nu(x,r) = \frac{|\omega_L^\infty \ast (\nabla \varphi)_r(x)|^2}{|\omega_L^\infty \ast \varphi_r(x)|^2} \frac{dx \, dr}{r}\]
is a Carleson measure. Since $\hm_L^\infty$ is a doubling measure (see Lemma \ref{gfinftyfacts.lemma}), Theorem \ref{FKPthrm.thrm} implies that $\hm_L^\infty$ is an $A_\infty(dx)$ measure, that is, $k^\infty_L = \tfrac{d\hm_L^\infty}{dx}$ exists and is an $A_\infty$ weight. It follows that $\log k^\infty_L \in BMO(\rn)$ \cite{GCRDF}. Moreover, since $A$ satisfies the vanishing weak DKP condition Theorem \ref{mainquantest.thrm} implies that $\nu$ is a Carleson measure with vanishing trace. Therefore, by \cite[Theorem 1]{Kor} $\log k^\infty_L \in VMO(\rn)$ (see Theorem \ref{Koreyequiv.thrm} (4) implies (5)).
\end{proof}

We observe that in the proof above, we did not use the {\it vanishing} assumption in the weak DKP condition to obtain that $\hm_L^\infty$ is an $A_\infty$ weight and that 
$\log k^\infty_L \in BMO(\rn)$. In particular, we have also proven the following theorem.
\begin{theorem}\label{KPalt.thrm}
Let $L = -\div A\nabla$ be a divergence form elliptic operator on $\ree_+$, whose coefficient matrix $A$ satisfies the weak DKP condition. Let $\hm^\infty_L$ be the elliptic measure with pole at infinity. Then $\hm^\infty_L \ll \mathcal{L}^n$, 
$\hm_L^\infty \in A_\infty(dx)$, and $k^\infty_L(y) := \tfrac{d\hm^\infty_L}{dx}(y)$ has the property that $\log k \in BMO(\rn)$. The implicit constants in the statements $\hm_L^\infty \in  A_\infty$ and $\log k \in BMO(\rn)$ are each bounded by a constant depending on $n$, ellipticity and 
$\left\|\alpha_2(x,r)^2 \, \frac{dx\, dr}{r}\right\|_{\mathcal{C}}$.
\end{theorem}

In order to prove the second half of Theorem \ref{mainvmo.thrm} (i.e. for finite poles), we need to move the pole in Theorem \ref{mainvmopt1.thrm} from infinity to a point $X \in \ree_+$. To do so we need some standard estimates for the quotient of solutions to divergence form elliptic equations. Most of these estimates can be found in \cite{JKNTA, KenigCBMS}, where they are stated for harmonic functions or operators with symmetric coefficients. But upon inspection the proofs do not rely on the symmetry when we use the appropriate notion of Green's functions for $L $ (see Definition \ref{def:Green} and also \cite{HMTbook}). 

\begin{lemma}[Comparison principle]\label{comparisonprinc.lem}
Let $L$ be a divergence form elliptic operator with ellipticity $\Lambda$ and $x \in \rn$ and $r > 0$.. If $Lu =Lv = 0$, $u, v \ge 0$ in $T(x,2r)$, $u$ and $v$ are non-trivial functions which vanish continuously on $\Delta(x, 2r)$ then
\[ 
\frac{u(X)}{v(X)} \approx \frac{u((x,r))}{v((x,r))}, \quad \forall X \in T(x,r),
\]
where the implicit constants depend on $n$, $\Lambda$.

\end{lemma}

\begin{lemma}[Quotients of non-negative solutions]\label{QHC.lem}
Let $L$ be a divergence form elliptic operator with ellipticity $\Lambda$ and $x \in \rn$ and $r > 0$. If $Lu =Lv = 0$, $u, v \ge 0$ in $T(x,2r)$, and $u$ and $v$ vanish continuously on $\Delta(x, 2r)$, then $u/v$ is H\"older continuous of order $\gamma = \gamma(n, \Lambda)$ in $\overline{B(x,r) \cap \ree_+}$. In particular, $\lim_{Y \to y} (u/v)(Y)$ exists\footnote{Here the limit is taken within $\ree_+$.} and, moreover,
\begin{equation}\label{QHC.eq}
\left| \frac{u(X)}{v(X)} -  \frac{u((x,r))}{v((x,r))} \right| \le C\left(\frac{|X - x|}{r}\right)^\gamma \frac{u((x,r))}{v((x,r))}, \quad \forall X, Y \in T(x,r),
\end{equation}
where the constant $C >0$ and $\mu \in (0,1)$ depend on $n$, $\Lambda$.
\end{lemma}

\begin{lemma}[Kernel function {\cite{KenigCBMS}}]\label{Kernelcomparehc.lem}
Let $L$ be a divergence form elliptic operator with ellipticity $\Lambda$ and $x \in \rn$ and $r > 0$. For every $X_0 = (x_0, t_0)$ and $X_1 = (x_1, t_1)$ there exists a kernel function $H(X_0,X_1, z)$ (a function of $z$) defined by
\[H(X_0,X_1, z) : = \frac{d\hm_L^{X_0}}{d\hm_L^{X_1}}(z).\]
The kernel function is given by 
\[H(X_0,X_1, z)  = \lim_{Z \to z} \frac{G_L(X_0,Z)}{G_L(X_1,Z)},\]
where the limit is taken inside of $\ree_+$. In particular, by Lemmas \ref{comparisonprinc.lem} and \ref{QHC.lem} for $t' := \min\{t_0,t_1\}/10$ and $|z - z'| < t'$ it holds 
\[|H(X_0,X_1, z) - H(X_0,X_1, z')| \le C\frac{G_L(X_0, (z, t')) }{G_L(X_1, (z, t'))}\left(\frac{|z - z'|}{t'} \right)^\gamma\]
and 
\[C^{-1}\frac{G_L(X_0, (z, t')) }{G_L(X_1, (z, t'))} \le H(X_0,X_1, z) \le C\frac{G_L(X_0, (z, t')) }{G_L(X_1, (z, t'))}  .\]
Here $C$ and $\gamma$ depend on $n$ and $\Lambda$.
\end{lemma}

We need a kernel function that takes $X_1$ to infinity, in order to compare elliptic measure with fixed poles to that with pole at infinity. We produce this function with the following lemma.
\begin{lemma}[Kernel function with infinite argument]\label{infkernfnlem.lem}
Let $L$ be a divergence form elliptic operator with ellipticity $\Lambda$. Given $X_0 = (x_0, t_0) \in \ree_+$ the kernel function (as a function of $z$) 
\[H_\infty(X_0,z) = \frac{d\hm^{X_0}_L}{d\hm^{\infty}_L} (z)\]
exists as a locally H\"older continuous function of order $\gamma = \gamma(n,\Lambda)$. Moreover, for every $\kappa > 1$ there exists $C_\kappa = C_\kappa(\kappa, n, \Lambda)$ such that
\begin{equation}\label{HCinfkern.eq}
|H_\infty(X_0, z) - H_\infty(X_0, z')| \le C_\kappa \frac{G_L(X_0, (x_0, t_0/4))}{U(X_0)} \left(\frac{|z - z'|}{t_0}\right)^\gamma 
\end{equation}
for all $z,z' \in \Delta(x_0, 5\kappa t_0)$, $|z - z'| < t_0/4$ and
\begin{equation}\label{compinfkern.eq}
(C_\kappa)^{-1} \frac{G_L(X_0, (x_0, t_0/4))}{U(X_0)} \le H_\infty(X_0, z) \le C_\kappa \frac{G_L(X_0, (x_0, t_0/4))}{U(X_0)}.
\end{equation}
for all $z \in \Delta(x_0, 5\kappa t_0)$. Here $U$ is the Green function for $L$ with pole at infinity, see Lemma \ref{def:Greeninf}.
\end{lemma}

\begin{proof}
We drop the subscript $L$ from the Green function and the elliptic measure throughout the proof. Let $\kappa > 1$. We recall that $U(Y) = \lim_{j \to \infty} u_{k_j}(Y) :=\lim_{j \to \infty} \frac{G((0,2^{k_j}), Y)}{G((0,2^{k_j}), (0,1))}$ and $\hm^\infty$ is the weak limit of $\frac{\hm^{(0,2^{k_j})}}{G((0,2^{k_j}), (0,1))}$. Set $X_j := (0,2^{k_j})$, $u_j = u_{k_j}$ and $\hm_j = \frac{\hm^{(0,2^{k_j})}}{G((0,2^{k_j}), (0,1))} = \frac{\hm^{X_j}}{G((0,2^{k_j}), (0,1))} $. We will often use that the $t$-coordinate of $X_j$ tends to infinity. 

Now, by Lemma \ref{Kernelcomparehc.lem},
\[\frac{d\hm^{X_0}}{d\hm^{X_j}}(z) = H(X_0,X_j, z)\]
exists as a H\"older continuous function. Multiplying, by $G(X_j, (0,1))$  ($=G((0,2^{k_j}), (0,1))$) we have 
\[\frac{d\hm^{X_0}}{d\hm_j}(z) = G(X_j, (0,1))H(X_0,X_j, z) =: H_j(X_0,z)\]
is locally H\"older continuous. More specifically, provided $2^{k_j} > t_0$, Lemmas \ref{comparisonprinc.lem} gives the estimates
\[ |H_j(X_0,z) - H_j(X_0,z')| \le M_j\left(\frac{|z - z'|}{t_0} \right)^\gamma \]
for $z, z' \in \Delta(x_0, 10\kappa t_0)$, $|z - z'| < t_0/4$
and 
\[m_j \le |H_j(X_0,z)| \le M_j\]
for $z \in \Delta(x_0, 10\kappa r)$,
where $M_j$ and $m_j$ depend on $\kappa$, $j$, $n$ and $\Lambda$ and are defined by
\[M_j := \sup\left\{ \frac{G(X_j, (0,1))}{G(X_j, (y,t_0/4))} G(X_0, (y,t_0/4)): y \in \Delta(x_0, 10\kappa t_0)\right \} \]
and 
\[m_j := \inf\left\{ \frac{G(X_j, (0,1))}{G(X_j, (y,t_0/4))} G(X_0, (y,t_0/4)): y \in \Delta(x_0, 10\kappa t_0) \right\}.\]
By the Harnack inequality 
\[M_j  \approx m_j \approx  \frac{G(X_j, (0,1))}{G(X_j, (x_0,t_0/4))} G(X_0, (x_0,t_0/4))  \]
where the implicit constants depend only on $\kappa$, $n$, $\Lambda$, but not on $j$. On the other hand,
\[\lim_{j \to \infty} \frac{G(X_j, (0,1))}{G(X_j, (x_0,t_0/4))} G(X_0, (x_0,t_0/4))  = \frac{G_L(X_0, (x_0, t_0/4))}{U(x_0,t_0/4))} \approx   \frac{G_L(X_0, (x_0, t_0/4))}{U(X_0)}.\]
Thus, there exists $C_\kappa$ such that for all sufficiently large $j$
\begin{equation}\label{HCinfkernj.eq}
|H_j(X_0, z) - H_j(X_0, z')| \le C_\kappa \frac{G_L(X_0, (x_0, t_0/4))}{U(X_0)} \left(\frac{|z - z'|}{t_0}\right)^\gamma 
\end{equation}
for all $z,z' \in \overline{\Delta(x_0, 5\kappa t_0)}$, $|z - z'| < t_0/4$ and
\begin{equation}\label{compinfkernj.eq}
 (C_\kappa)^{-1} \frac{G_L(X_0, (x_0, t_0/4))}{U(X_0)} \le H_j(X_0, z) \le C_\kappa \frac{G_L(X_0, (x_0, t_0/4))}{U(X_0)}.
\end{equation}
for all $z \in \overline{\Delta(x_0, 5\kappa t_0)}$. Since $\kappa > 1$ is arbitrary, we may use the Arzela-Ascoli theorem to produce a subsequence  $H_{j_m}(X_0, z)$ converging to a function $H_\infty(X_0,z)$ locally uniformly and such that for fixed $\kappa > 1$ estimates \eqref{HCinfkernj.eq} and \eqref{compinfkernj.eq} hold. Moreover, by definition of $H_j$
\begin{equation}\label{Hiotherway.eq}
d \hm^{X_0} = H_j(X_0,z) d\hm_j
\end{equation}
and since $\hm_j \rightharpoonup \hm_\infty$ it holds that 
\[d \hm^{X_0} = H_\infty(X_0,z) d\hm_\infty.\]
Indeed, fix $f \in C_c(\rn)$ and $R > 0$ so that $\supp f \subset \Delta(0,R)$. Note that $fH_\infty(X_0,\cdot) \in C_c(\rn)$. Then by \eqref{Hiotherway.eq} and the fact that $H_{j_m}(X_0, z)$ converges locally uniformly to $H_\infty(X_0,z)$
\begin{align*}
&\int_{\rn} f d\hm^{X_0} \, dx = \lim_{m \to \infty} \int_{\rn} f H_{j_m}(X_0,z) d\hm_{j_m}
\\& = \lim_{m \to \infty} \int_{\rn} [f H_\infty(X_0,z)] d\hm_{j_m} + \lim_{m \to \infty} \int_{\rn} [f (H_{j_m}(X_0,z) - H_\infty(X_0,z))] d\hm_{j_m}
\\ & = \int_{\rn} f H_\infty(X_0,z) \, d\hm^\infty + 0,
\end{align*}
where, to show the second limit in the second line was zero, we used 
\begin{multline*}
\left| \int_{\rn} [f (H_{j_m}(X_0,z) - H_\infty(X_0,z))] d\hm_{j_m} \right|
\\  \le \|f\|_\infty \sup_{z \in \Delta(0,R)}|H_{j_m}(X_0,z) - H_\infty(X_0,z))| \hm_{j_m}(\Delta(0,R))
\end{multline*}
and that $\hm_{j_m}(\Delta(0,R))$ is uniformly bounded  (in $m$) for sufficiently large $m$. To see the later fact, we first write
\[\hm_{j_m}(\Delta(0,R))) = \frac{\hm^{X_{j_m}}(\Delta(0,R))}{G(X_{j_m}, (0,1))}\]
then the CFMS estimates (Lemma \ref{CFMS.lem}) and the Harnack inequality give
\[\frac{\hm^{X_{j_m}}(\Delta(0,R))}{G(X_{j_m}, (0,1))} \lesssim R^{n-1} \frac{G(X_{j_m}, (0,R))}{G(X_{j_m}, (0,1))} \lesssim C_R\]
for sufficiently large $m$.
As all of the desired properties of $H_\infty(X_0,z)$ have been demonstrated, this proves the lemma.
\end{proof}

\begin{theorem}\label{mainvmopt2.thrm}
Let $L = -\div A\nabla$ be a divergence form elliptic operator on $\ree_+$, whose coefficient matrix $A$ satisfies the vanishing weak DKP condition. Let $\hm^{X_0}_L$ be the elliptic measure with pole at ${X_0}$. Then $\hm^{X_0}_L \ll \mathcal{L}^n$, and $k^{X_0}_L(y) := \tfrac{d\hm^{X_0}_L}{dx}(y)$ has the property that $\log k^{X_0}_L \in VMO_{loc}(\rn)$. 
\end{theorem}

\begin{proof}
To ease notation we drop the subscript $L$ in the proof. 
%We intend on using Theorems \ref{Koreyequiv.thrm} and \ref{VMOlocKorey.thrm} to prove this theorem. 
By Theorem \ref{mainvmopt2.thrm} $\hm^\infty \ll \mathcal{L}^n$ and $k^\infty := \frac{d\hm^\infty}{dx}$ satisfies $\log k^\infty \in VMO$. Then by Theorem \ref{Koreyequiv.thrm} it holds that
\begin{equation}\label{kinftyisVMORH2.eq}
\lim_{r_0 \to 0^+} \sup_{x \in \rn} \sup_{r \in (0,r_0)} \frac{\left(\fint_{\Delta(x,r)} (k^\infty(z))^2 \, dz \right)^{1/2}}{\fint_{\Delta(x,r)}  k^\infty(z) \, dz}  = 1.
\end{equation}
It then follows from Lemma \ref{infkernfnlem.lem} that $k^{X_0} = \frac{d \hm^X}{dx}$ exists $\mathcal{L}^n$ a.e. with
\[k^{X_0}(z)= \frac{d\hm^{X_0}}{dx}(z)  = \frac{d\hm^{X_0}}{d\hm^\infty}(z)\frac{d\hm^\infty}{dx}(z)  = H_\infty(X_0, z) k^\infty(z).\]
By Theorem \ref{VMOlocKorey.thrm} it suffices to show that for fixed $R > 0$ and $\epsilon > 0$ there exists $r_0 > 0$ such that
\begin{equation}\label{finitepoleVMOgoal.eq}
\sup_{x \in \Delta(0,R)} \sup_{r \in (0,r_0)}  \frac{\left(\fint_{\Delta(x,r)} (k^{X_0}(z))^2 \, dz \right)^{1/2}}{\fint_{\Delta(x,r)} k^{X_0}(z) \, dz} \le (1 + \epsilon)^2.
\end{equation}
To this end, let $R, \epsilon > 0$ be fixed. By \eqref{kinftyisVMORH2.eq} there exists $r_1$ such that 
\begin{equation}\label{kinftyisVMORH2rone.eq}
\sup_{x \in \rn} \sup_{r \in (0,r_1)} \frac{\left(\fint_{\Delta(x,r)} (k^\infty(z))^2 \, dz \right)^{1/2}}{\fint_{\Delta(x,r)}  k^\infty(z) \, dz} \le 1 + \epsilon.
\end{equation}
Write $X_0 = (x_0,t_0)$ and let $\kappa$ be large enough (depending on $R$) so that $\Delta(0,10R) \subset \Delta(x_0, \kappa t_0)$. By Lemma \ref{infkernfnlem.lem}, that is, estimates \eqref{HCinfkern.eq} and \eqref{compinfkern.eq} there exists a constant $C'$, depending on $\kappa$, $n$ and ellipticity, such that for $z, z' \in \Delta(x_0, 5\kappa t_0)$ with $|z - z'| < t_0/4$
\begin{equation}\label{infkerncombo.eq}
\left| \frac{H_\infty(X_0, z)}{H_\infty(X_0, z')} - 1\right| = \frac{1}{H_\infty( X_0, z')} \left|H_\infty(X_0, z) - H_\infty(X_0, z') \right| \lesssim C' \left(\frac{|z - z'|}{t_0}\right)^{\gamma}. 
\end{equation}
Let $r_2 \in (0, t_0/4) $ be such that $C' \left(\frac{2 r_2}{t_0}\right)^{\gamma} < \epsilon$. Then for $r <  \min\{r_1,r_2, 9R\} =: r_0$ and $x \in \Delta(0,R)$ it holds
\begin{align*}
 \frac{\left(\fint_{\Delta(x,r)} (k^{X_0}(z))^2 \, dz \right)^{1/2}}{\fint_{\Delta(x,r)} k^{X_0}(z) \, dz} &=  \frac{\left(\fint_{\Delta(x,r)} (H_\infty(X_0, z) k^\infty(z))^2 \, dz \right)^{1/2}}{\fint_{\Delta(x,r)} H_\infty(X_0, z) k^\infty(z) \, dz}
 \\ & \le \frac{\sup_{z \in \Delta(x,r)}H(z)}{\inf_{z \in \Delta(x,r)}H(z')} \frac{\left(\fint_{\Delta(x,r)} (k^\infty(z))^2 \, dz \right)^{1/2}}{\fint_{\Delta(x,r)}  k^\infty(z) \, dz}
 \\ & \le (1 + \epsilon)^2,
 \end{align*}
where we used the choice of $r_0$, \eqref{infkerncombo.eq} and \eqref{kinftyisVMORH2rone.eq}. This shows \eqref{finitepoleVMOgoal.eq} and hence $\log k^{X_0} \in VMO_{loc}(\rn)$.
\end{proof}

Combining Theorems \ref{mainvmopt1.thrm} and \ref{mainvmopt2.thrm} gives Theorem \ref{mainvmo.thrm}.

\section{Globalizing local DKP conditions and the works of Kenig and Pipher and Dindos, Petermichl and Pipher}\label{kpdpp.sect}

In this section, we reflect on the relationship between our results and the related works \cite{KP,DPP}. %These works have been quite influential and since our estimates in Theorem \ref{mainquantest.thrm} are quantitative, we have begun to wonder whether how/if they fit into the framework here. 
The results in \cite{KP,DPP} are for Lipschitz domains, which requires one to obtain localized estimates; however, our Theorems \ref{mainvmopt1.thrm} and \ref{mainvmopt2.thrm} are for operators that are defined globally in a half space. To bridge the gap we show how to extend coefficients satisfying a local weak DKP condition to globally defined coefficients. We then proceed to show a local version of Theorem \ref{KPalt.thrm}, which was originally shown in \cite{KP}, under the hypothesis of a gradient condition on the coefficients. This condition trivially controls the weak DKP coefficients, that is, the $\alpha_2$-numbers (and also the $\alpha_\infty$-numbers, defined in \eqref{def:oscinf} below).

The result of \cite{DPP} is a small constant version of \cite{KP}. To be precise, in \cite{DPP} the authors show that the $L^p$-Dirichlet problem is solvable for any fixed $p > 1$ {\it provided} the Carleson norm in a related weak DKP-type condition is sufficiently small. In \cite{DPP} the authors use a condition that is comparable to using $\alpha_\infty$ coefficients defined by
\begin{equation}\label{def:oscinf}
	\alpha_\infty(x,r) := \inf_{A_0 \in \mathfrak{A}(\Lambda)}\sup_{(y,s) \in W(x,r)} |A(y,s) - A_0|.
\end{equation} 
(Note that $\alpha_2(x,r)  \le \alpha_\infty(x,r)$ so that the coefficients used in the current work are controlled by those in \cite{DPP}.)
Equivalently, they show a local $L^{p'}$ reverse-H\"older inequality for the elliptic kernel, under this smallness assumption. Recent work by the first author with Egert and Saari \cite{BES} seems to indicate that our main quantitative estimate, Theorem \ref{mainquantest.thrm}, provides an alternative approach to their result (and control on the constant in the reverse H\"older inequality) in the upper half-space. We discuss this briefly in Section \ref{dppsect.sect}.

\subsection{Extending local DKP conditions and an alternative approach to Kenig and Pipher's theorem}\label{kpsecondlook.sect} In this subsection we show how to extend coefficients $A$ satisfying the weak DKP condition on a Carleson region $T(x,r)$ to coefficients $\widetilde{A}$ so that $\widetilde{A}$ agrees with $A$ on a smaller Carleson region $T(x,cr)$ and satisfies a global weak DKP condition. We also want to ensure that we do not increase the the constant in the weak DKP condition `too much'. Then we will show how to use this to show an analogue of the main result in \cite{KP}.

For the purposes of constructing these extensions we write $\alpha_2(x,r,A)$ and $\alpha_2(x,r, \widetilde{A})$ to denote the $\alpha_2$ coefficients for $A$ and $\widetilde{A}$ respectively. For instance,
\[\alpha_2(x,r,\widetilde{A}) = \inf_{A_0 \in \mathfrak{A}(\Lambda)}\left(\fiint_{(y,s) \in W(x,r)} |\widetilde{A}(y,s) - A_0|^2 \right)^{1/2}.\]
We also define for $x \in \rn$ and $r > 0$ the cylindrical region
\[\Gamma(x,r) = \Delta(x,r) \times (0, r).\]

\begin{lemma}\label{localizeDKP.lem}
Let $A$ be a matrix that satisfies the weak DKP condition on $T(x_0,R_0)$ for some $x_0 \in \rn$ and $R_0> 0$.  There exists $\widetilde{A}$ such that $\widetilde{A} = A$ on $\Gamma(x_0, cR_0)$ and $\widetilde{A}$ satisfies the weak DKP condition on $\ree_+$, where $c$ is an absolute constant. Moreover, we have the estimates
\begin{equation}\label{localdkplem1.eq}
\begin{split}
\left\| \alpha_2(x,r,\widetilde{A})^2 \, \frac{dx \, dr}{r}\right\|_{\mathcal{C}} &\lesssim \left\| \alpha_2(x,r,A)^2 \, \frac{dx \, dr}{r}\right\|_{\mathcal{C}(\Delta(x_0,R_0))}  
\\ & \quad + \min\left\{1, \left\| \alpha_2(x,r,A)^2 \, \frac{dx \, dr}{r}\right\|_{\mathcal{C}(\Delta(x_0,R_0))}^{4/(n+3)}  \right\} 
\end{split}
\end{equation}
and for $r_0 < R_0$
\begin{equation}\label{localdkplem2.eq}
\sup_{x \in \rn} \left\| \alpha_2(x,r,\widetilde{A})^2 \, \frac{dx \, dr}{r}\right\|_{\mathcal{C}(\Delta(x,r_0))} \lesssim \sup_{\substack{\Delta \subset \Delta(x_0,R_0) \\ r(\Delta) \le r_0}}\left\| \alpha_2(x,r,A)^2 \, \frac{dx \, dr}{r}\right\|_{\mathcal{C}(\Delta(x_0,R_0))}  + (r_0/R_0)^2, 
\end{equation}
where the implicit constants depend on dimension and ellipticity. 
\end{lemma}
\begin{proof}
By translation we may assume $x_0 = 0$ and we set \[C_A: = \left\| \alpha_2(x,r,A)^2 \, \frac{dx \, dr}{r}\right\|_{\mathcal{C}(\Delta(x_0,R_0))}.\] We choose $c$ an absolute constant so that $\Gamma(0, 10^{10}cR_0) \subset T(x,R_0)$. Set $R := 2cR_0$ and let $A_0$ be the constant coefficient matrix so that
\[\gamma(0,50R,A) = \left(\fiint_{T(0,50R)} |A(y,s) - A_0|^2 \, dy \, ds\right)^{1/2}. \]
Note that the point-wise inequality on $\gamma(x,r)$ in Lemma \ref{gammabdalpha.lem} gives
\begin{equation}\label{gammaptws.eq}
\gamma(0,50R,A) \lesssim C_A.
\end{equation}

Now we set
\[\widetilde{A}(y,s) = \mathbbm{1}_{\Gamma(0,R)}(y,s) \left[(1- f(|y|))A(y,s) + f(|y|)A_0 \right] +\mathbbm{1}_{\left(\Gamma(0,R)\right)^c}(y,s)  
%(1-  \mathbbm{1}_{\Gamma(0,R)}(y,s))
A_0,\]
where $f:[0,\infty) \to [0,1]$ is the piece-wise defined function
\begin{equation}\label{def:f}
	f(a) := \begin{cases}
0 & \text{ if } a \in [0,R/2] \\
(2/R)(a - R/2) & \text{ if } a \in (R/2,R] \\
1 & \text{ if } a > R.
\end{cases}
\end{equation} 
Note that $f$ is ($2/R$)-Lipschitz. By inspection we see that $\widetilde{A} = A$ in $\Gamma(0,cR_0) = \Gamma(0, R/2)$, so we only need to verify the estimates on the Carleson norms. 

To do this, we let $(x,r) \in \ree_+$ and estimate $\alpha_2(x,r, \widetilde{A})$. We break our analysis up into cases and combine them later. \\
{\bf Case 0:} $W(x,r)$ does not meet $\Gamma(0,R)$. In this case $A(y,s) = A_0$ a constant, $\Lambda$-elliptic matrix in $W(x,r)$. Thus, $\alpha_2(x,r, \widetilde{A}) = 0$.
\\ {\bf Case 1:} $W(x,r)$ is contained in $\Gamma(0,R)$. In this case $r \le R$ and we let $A_{x,r}$ be a constant, $\Lambda$-elliptic matrix such that
\[\alpha(x,r, A) =   \left(\fiint_{W(x,r)} |A(y,s) - A_{x,r}|^2 \, dy \, ds\right)^{1/2}.\]
Now set 
\[\widetilde{A}_{x.r} : = [1 - f(|x|)]A_{x,r} -f(|x|) A_0,\]
a constant, $\Lambda$-elliptic matrix
We make the estimate
\begin{align*}
|\widetilde{A}(y,s) - \widetilde{A}_{x,r}| & \le |[1 - f(|y|)]A(y,s) - [1 - f(|x|)]A_{x,r}| 
\\ & \quad + |[f(|x|) - f(|y|)]A_0|
\\ & \le |[1 - f(|x|)](A(y,s) - A_0)| + |f(|x|) - f(|y|)|(|A_0| + |A_{x,r}|)
\\ & \le |A(y,s) - A_0| + \frac{8r}{R} \Lambda,
\end{align*}
where we used that $f(|x|) \in [0,1]$, $f$ is ($2/R$)-Lipschitz, $|x - y| < 2r$ and $A_0$ and $A_{x,r}$ are $\Lambda$-elliptic. By using $\widetilde{A}_{x.r}$ in the definition of $\alpha_2(x,r, \widetilde{A})$ we obtain
\[\alpha_2(x,r, \widetilde{A}) \le \alpha_2(x,r, A) +  \frac{8r}{R} \Lambda.\]
\\
{\bf Case 2:} $W(x,r)$ meets both $\Gamma(0,R)$ and $\Gamma(0,R)^c$. In this case we take $\widetilde{A}_{x,r} = A_0$ and we estimate
\begin{equation}\label{case2rough.eq}
|\widetilde{A}(y,s) -\widetilde{A}_{x,r}| = \mathbbm{1}_{\Gamma(0,R)}(y,s)|[1 - f(|y|)]A(y,s) - [1 - f(|y|)]A_0|
\end{equation}
We break into further cases, setting $M = \max\{100, C_A^{-2/(n+3)}\}$.
\\
{\bf Case 2a:} $W(x,r)$ meets both $\Gamma(0,R)$ and $\Gamma(0,R)^c$ and $r > R/M$. If $M = 100$ we use that $|1 - f(|y|)| \le 1$ and that $A$ and $A_0$ are $\Lambda$-elliptic to deduce from \eqref{case2rough.eq} that 
\[|\widetilde{A}(y,s) -\widetilde{A}_{x,r}| \le 2\Lambda \le 200\Lambda(r/R)\]
Otherwise, $M = C_A^{-2/(n+3)}$ and we deduce from \eqref{gammaptws.eq}
\begin{align*}
\alpha(x,r, \widetilde{A}) =  |W(x,r)|^{-1/2} \left( \iint_{W(x,r) \cap \Gamma(0,R)} |A(y,s) - A_0|^2 \right)^{1/2} \lesssim (R/r)^{(n+1)/2}(C_A)^{1/2}. 
\end{align*}
Here we used that $W(x,r) \subset T(0,50R)$ since $W(x,r)$ meets both $\Gamma(0,R)$ implies $r/2 < R$.
\\ {\bf Case 2b:} $W(x,r)$ meets both $\Gamma(0,R)$ and $\Gamma(0,R)^c$ and $r \le R/M$. In this case, $W(x,r)$ must meet the `side' of $\partial \Gamma(0,R)$ since $W(x,r) \cap \{(y,R): y \in \rn\} = \emptyset$. Then there exists $y_0 \in W(x,r)$ with $|y_0| = R$ and hence for $(y,s) \in W(x,r)$ we have
\[|1 - f(|y|)| \le |1 - f(|y_0|)| + |f(|y_0|)  - f(|y|)| \le 0 + 6r/R,\]
where we used that $f$ is ($2/R$)-Lipschitz. Thus, using \eqref{case2rough.eq} we find
\[|\widetilde{A}(y,s) -\widetilde{A}_{x,r}| \le \frac{12r}{R} \Lambda,\]
where we used $A$ and $A_0$ are $\Lambda$-elliptic.
\\ 
Combining the the cases we have (by choice of $c$)
\[\alpha_2(x,r, \widetilde{A}) \le C\mathbbm{1}_{T(0,10^{-3}R_0)}(x,r) \left[\alpha(x,r, A) + h(r,R, C_A)\right],\]
where $C$ depends on dimension and ellipticity and the function $h(r,R,C_A)$ is given by 
\[ h(r,R,C_A) = \begin{cases}
 (R/r)^{(n+1)/2}  (C_A)^{1/2} & \text{ if } r > R/M \\ 
  (r/R) & \text{ if } r \le R/M
\end{cases}
\] 
if $M  = (C_A)^{-2/(n+3)}$  and $h(r,R,C_A) = 200\Lambda(r/R)$ if $M = 100$. (Recall $M = M(C_A) = \max\{100, (C_A)^{-2/(n+3)}\}$.)
 The Carleson measure bounds follow from this estimate and this proves the Lemma.
\end{proof}

Now that we have a matrix satisfying the (global) weak DKP condition, we can apply the results of Section \ref{mtproof.sect} to $\widetilde{A}$; however, we want to say things about the elliptic measure/kernel for $L = -\div A\nabla$, not $\widetilde{L} = -\div \widetilde{A} \nabla$. The following lemma allows us to pass estimates on $k_{\widetilde{L}}$ to $k_{L}$, albeit in a rough manner. \footnote{In some applications, for example as we will discuss in the Section \ref{dppsect.sect}, we need a sharper, more quantitative estimate of $d\omega_1^{X_0}/d\omega_2^{X_0}$. In that case we appeal instead to \cite[Lemma 5.1 and Remark 5.11]{BTZ}.}

\begin{lemma}[{\cite[Lemma 2.22]{BTZ}}]\label{roughcompdifop.lem}
Suppose that $L_i = -\div A_i \nabla$, $i = 1,2$ are two divergence form $\Lambda$-elliptic operators with $A_1 = A_2$ on $T(x,10r)$. Let $\hm_i^{X_0}$ is the $L_i$-elliptic measure with pole at $X_0\in  T(x,8r) \setminus T(x,4r)$.
Then $\hm_1^{X_0}|_{\Delta(x,r)}$ and $\hm_2^{X_0}|_{\Delta(x,r)}$ are mutually absolutely continuous. In particular, if $\hm_1^{X_0}|_{\Delta(x,r)}$ and $\mathcal{L}^n|_{\Delta(x,r)} $ are mutually absolutely continuous then so are $\hm_2^{X_0}|_{\Delta(x,r)}$ and $\mathcal{L}^n|_{\Delta(x,r)} $. 
Moreover,
\[\frac{d\hm_1^{X_0}}{d\hm_2^{X_0}}(y) \approx 1, \quad \hm_2^{X_0}-a.e.\ y \in \Delta(x,r/2),\]
where the implicit constants depend on $n$, $\Lambda$.
\end{lemma}

In what follows, for $\Delta = \Delta(x,r)$ we set $X_{\Delta} = (x, 12r)$. We give an alternative proof to the following quantitative result from \cite{KP}.
\begin{theorem}\label{KPRThrm.eq}
Suppose that $A$ satisfies the weak DKP condition on $T(x_0,R_0)$ for some $x_0 \in \rn$ and $R_0  > 0$, that is,
\[\left\| \alpha_2(x,r,A)^2 \, \frac{dx \, dr}{r}\right\|_{\mathcal{C}(\Delta(x_0,R_0))} =: C_A < \infty.\] 
Then there is a absolute constant $c'$ such that for every $\Delta$ such that $\Delta \subset \Delta(x_0, c'R_0)$ it holds 
\begin{equation}\label{KPRHP.eq}
\left(\fint_{\Delta} (k^{X_\Delta}_L(z))^p \,dz \right)^{1/p} \le C \fint_{\Delta} k^{X_\Delta}_L(z) \, dz,
\end{equation}
where the constants $C, p > 1$ depend on $n$, $\Lambda$ and $C_A$.
\end{theorem}
The estimate \eqref{KPRHP.eq} is often what is referred to as the `reverse Holder inequality for the elliptic kernel' and is equivalent to a local $A_\infty$-type condition for the elliptic measure. If \eqref{KPRHP.eq} holds for all $\Delta$, we have that the $L^{p'}$ Dirichlet problem is solvable on $\ree_+$ (see e.g. \cite[Proposition 4.5]{HofPLe}). The flexibility of our `globalization' lemma (Lemma \ref{localizeDKP.lem}) and a `standard' pullback mechanism\footnote{This pullback is called sometimes referred to as the Dahlberg-Kenig-Stein pullback, see \cite{KP, Dahlpullback}.} allows one extend the above theorem to the setting of Lipschitz domains as was done by Kenig and Pipher \cite{KP}. 

\begin{proof}[Proof of Theorem \ref{KPRThrm.eq}]\label{KPreprovefull.thrm}
By translation we may assume $x_0 = 0$. Let $A$ be as in the statement of the theorem. Let $\widetilde{A}$ be the operator produced by Lemma \ref{localizeDKP.lem} and $\widetilde{L} = -\div \widetilde{A} \nabla$. Then $\widetilde{A}$ satisfies the {\it global} weak DKP condition with
\[\left\| \alpha_2(x,r, \widetilde{A})^2 \, \frac{dx \, dr}{r} \right\|_{\mathcal{C}} \lesssim C_A + 1.\]
By Theorem \ref{KPalt.thrm},  $\hm^\infty_{\widetilde{L}} \in A_\infty$, where $\hm^\infty_{\widetilde{L}}$ is the elliptic measure for $\mathcal{L}$ with pole at infinity. Moreover, the constants in the $A_\infty$ condition are controlled by $C_A$, $n$ and ellipticity. It follows from basic properties of $A_\infty$ weights (see Definition \ref{Ainfty.def}) that 
%$\hm^\infty_{\widetilde{L}}  \ll \mathcal{L}^n$ and 
the kernel $k^\infty_{\widetilde{L}} = \frac{d\hm^\infty_{\widetilde{L}}}{dx}$ satisfies a reverse H\"older condition, that is,
\[\left(\fint_{\Delta(x,r)} (k^\infty_{\widetilde{L}}(z))^p \, dz \right)^{1/p} \le \widetilde{C} \fint_{\Delta(x,r)} k^\infty_{\widetilde{L}}(z) \, dz, \quad \forall x \in \rn, r > 0,\]
where $\widetilde{C}, p > 1$ depend only on $n$, ellipticity and $C_A$. 

Next, we move the pole to a finite one. Using Lemma \ref{infkernfnlem.lem} we obtain\footnote{This time we do not need to be as precise as we were in Theorem \ref{mainvmopt2.thrm}. Instead, we only need the bound \eqref{compinfkern.eq}, and we do not use \eqref{HCinfkern.eq}.}
\begin{equation}\label{movedtofinitedifop.eq}
\left(\fint_{\Delta(x,r)} (k^{X_\Delta}_{\widetilde{L}}(z))^p \, dz \right)^{1/p} \le C' \fint_{\Delta(x,r)} k^{X_\Delta}_{\widetilde{L}}(z) \, dz,
\end{equation}
where $p >1$ is as above and $C'$ depends only on $n$, $\Lambda$ and $C_A$. To conclude, we need to change the operator $\widetilde{L}$ to $L$. Recall that in Lemma \ref{localizeDKP.lem} $\widetilde{A} = A$ on $\Gamma(0, cR_0)$. Now we choose $c'$ so small to ensure that $\Delta(x,r) \subset \Delta(0, c' R_0)$ implies $T(x,50r) \subset \Gamma(0, cR_0)$. This allows us to apply Lemma \ref{roughcompdifop.lem} to any such $\Delta(x,r)$, so that we obtain
\[k^{X_\Delta}_{\widetilde{L}}(z) \approx k^{X_\Delta}_{L}(z), \quad \forall \Delta \subset \Delta(x_0, c'R), z \in \Delta,\]
with constants depending on $n$ and ellipticity.
Using this estimate in conjunction with \eqref{movedtofinitedifop.eq} yields \eqref{KPRHP.eq} and proves the theorem.
\end{proof}

%{\color{red} \begin{remark}\label{KPlocalizationarg.rmk}
%The methods described in this subsection are intended to be used to transfer results on the half space to results on bounded Lipschitz domains, but they can also be used for the opposite purpose by `localizing' the coefficient matrix. For instance, if $A$ is a matrix-valued function on $\ree_+$ that satisfies the (global) weak DKP condition, then for any $X_0 = (x_0, t_0)$ and $\kappa > 100$ we can use Lemma \ref{localizeDKP.lem} to produce coefficients $\widetilde{A}$ which satisfy the weak-DKP condition in $\Gamma(x_0, 100 \kappa t_0)$, agree with $A$ on  $\Gamma(x_0, \kappa t_0)$ and are constant on $\Gamma(x_0, 100 \kappa t_0) \setminus \Gamma(x_0, 2 \kappa t_0)$. Then $\widetilde{A}$ will satisfy an appropriate DKP condition adapted to the bounded Lipschitz domain $\Gamma(x_0, 100 \kappa t_0)$. This allows one to use Lemma \ref{roughcompdifop.lem} and the main results in \cite{KP,DPP} (stated on Lipschitz domains) to conclude that $\hm_{L}^{X_0}$ is absolutely continuous with respect to surface measure on $\Delta(x_0, (\kappa/2) t_0)$.
%\end{remark}}

\subsection{Remarks concerning the work of Dindos, Petermichl and Pipher}\label{dppsect.sect}
In a few places we have made some remarks regarding how our work can be used to complement the work of \cite{DPP}. This subsection clarifies these remarks. In \cite{DPP} the authors show that for any fixed $p > 1$ the conclusion of Theorem \ref{KPreprovefull.thrm}, that is, \eqref{KPRHP.eq} holds {\it provided} $C_A$ therein is sufficiently small. We do not attempt to reprove this here. One reason is that we use \cite{FKP}, which gives {\it rough} bounds for the $A_\infty$ constant of a doubling weight in terms of the Carleson norm of the measure $\mu$ and its doubling constant as in Theorem \ref{FKPthrm.thrm}. In particular, the argument in \cite{FKP} as is written does not provide the estimates we ask for below. It should be noted that the work of Korey does {\it not} explicitly treat the small constant case, but rather the vanishing constant case, where the measure $\mu$ in Theorem \ref{FKPthrm.thrm} is a Carleson measure with vanishing trace (recall the definition in \eqref{def:vCarl}).

After this article was completed, the first author, Egert and Saari \cite{BES} showed the following. Suppose $w$ is a weight that is doubling with constant $C_{doub}$. For $\epsilon > 0$, there exists $ \delta = \delta(n, C_{doub},\epsilon) > 0$ such that if the measure $\mu$ in Theorem \ref{FKPthrm.thrm} with $\omega = w\, dx$ satisfies $\|\mu\|_{\mathcal{C}} \leq \delta$ then $[w]_{A_\infty} \leq 1 + \epsilon$.

This allows one to show Theorem \ref{KPRThrm.eq} for specific $p > 1$, {\it provided} the Carleson norm in the weak DKP condition is sufficiently small. In particular, this might give an alternative proof of the results in \cite{DPP}. Some indication of how to proceed to treat more general geometric settings, as was done in \cite{DPP}, is contained in \cite[Section 5]{BES}.

We did not investigate whether one can control `local $A_\infty$ constants' with a localized estimate on the Carleson norm of $\mu$, that is, $\|\mu\|_{\mathcal{C}(\Delta_0)}< \delta$ implies \eqref{A_inftydef.eq} with $C = (1 + \epsilon)$ on all balls $\Delta(x,r)$ sufficiently small and well contained in $\Delta_0$. This would allow one to use Theorem \ref{mainquantest.thrm} more directly to prove the Theorems that follow it. In particular, it would (conveniently) eliminate the need for using the pole change arguments to treat the finite pole case.

\appendix
\renewcommand{\theequation}{A.\arabic{equation}}
\section*{Appendix. An extension of Theorem \ref{mainvmo.thrm}}
\begin{theorem}\label{thm:squareDini}
	Let $\Omega$ be a $C^1$-square Dini domain in $\RR^{n+1}$. Let $A(\cdot)$ be an elliptic matrix in $\Omega$ which satisfies the weak DKP condition with vanishing trace. Then for any $X_0 \in \Omega$, the elliptic measure $\omega_\Omega^{X_0}$ is absolute continuous with respect to the boundary surface measure $ \sigma := \mathcal{H}^n|_{\pO}$, and moreover, the Poisson kernel $k_\Omega := \frac{d\omega_\Omega^{X_0}}{d\sigma}$ satisfies $\log k_\Omega  \in VMO_{loc}(\partial\Omega)$.
\end{theorem}

%Let $\Omega \subset \RR^{n+1}$ be the domain above the graph of $\varphi: \RR^n \to \RR$. Assume that $\varphi \in C^1$, and we denote the modulus of continuity of $\nabla \varphi$ by $\theta$, i.e. 
%\begin{equation}\label{eq:moc}
%	| \nabla \varphi(x) - \nabla \varphi(y)| \leq \theta(|x-y|), \quad \text{ for every } x, y \in \RR^n. 
%\end{equation} 
In general, when we say $\Omega$ is a $C^1$-square Dini domain, it means there exist $R>0$, finitely many boundary points $x_i \in \partial\Omega$ and cylindrical regions $\mathcal{C}(x_i, R)$\footnote{Modulo an orthogonal transformation $\mathcal{C}(x_i, r)$ is defined as $\{(x,t) \in \RR^n \times \RR: |x-x_i|< R, -f(R) R < t < f(R) R \}$ where $f(R) = \max\{1, 2\theta(R)\}$. The choice of $f(R)$ guarantees that the graph of $\varphi $ on the ball $\Delta(x_i, R))$ is completely contained in the cylinder.} centered at $x_i$ such that for each $i$, $\Omega \cap \mathcal{C}(x_i, R)$ is the region above the graph of a $C^1$-square Dini function $\varphi_i$. Assume without loss of generality that $X_0 \in \mathcal{C}(x_i, R) \setminus \mathcal{C}(x_i, R/2)$ for each $i$. It is not hard to see that (with the help of a cut-off function) we can extend $\varphi_i$ to a globally-defined function $\varphi: \RR^n \to \RR$, such that the modulus of continuity of $\nabla\varphi$ is also bounded above by $\theta$. Moreover, applying \cite[Lemma 5.1]{BTZ} (in particular, see \cite[Remark 5.11]{BTZ}) to $C^1$ domains, we can show that if two elliptic operators $L_1$ and $L_2$ agree in $\Omega \cap \mathcal{C}(x_i, R)$, then the ratio of their elliptic measures $d\omega_{L_2}^{X_0}/d\omega_{L_1}^{X_0}$ has small oscillation in a surface ball, whose radius is much smaller compared to $R$. In particular, it implies by a similar proof as that of Lemma \ref{mainvmopt2.thrm} that
\begin{equation}\label{eq:logVMO}
	\log k_{L_1} \in VMO(\partial\Omega \cap \mathcal{C}(x_i, R/2)) \iff \log k_{L_2} \in VMO(\partial\Omega \cap \mathcal{C}(x_i, R/2)). 
\end{equation} 
Therefore the proof of Theorem \ref{thm:squareDini} is reduced to the setting where $\Omega$ is the region above the graph of a single function $\varphi: \RR^n \to \RR$.

More precisely, let $\Omega \subset \RR^{n+1}$ be the domain above the graph of $\varphi: \RR^n \to \RR$, where the modulus of continuity for $\nabla \varphi$ satisfies the square Dini condition (see \eqref{def:squareDini}). 
%Assume that $\varphi \in C^1$, and we denote the modulus of continuity of $\nabla \varphi$ by $\theta$, i.e. 
%\begin{equation}\label{eq:moc}
%	| \nabla \varphi(x) - \nabla \varphi(y)| \leq \theta(|x-y|), \quad \text{ for every } x, y \in \RR^n. 
%\end{equation} 
%
% of a globally-defined $C^1$-square Dini domain $\Omega$ and a globally-defined elliptic operator satisfying the weak DKP condition with vanishing trace.
Assume without loss of generality that $\varphi(0) = 0$ and $\nabla \varphi(0) = 0$. Let $A(x,t)$ be a uniformly elliptic coefficient matrix in $\Omega$, which satisfies the weak DKP condition with vanishing trace, which means the following. 
For any $x_0 \in \RR^n$ fixed, we define the Whitney region
%\[ T_{\Omega} (x_0, r) := \Omega \cap B_r(x_0,\varphi(x_0)), \]
\[ W_\Omega(x_0, r) := \left\{(x,t) \in \RR^n \times \RR: x\in \Delta_r(x_0), ~\varphi(x)+\frac{r}{2} < t \leq \varphi(x) + r \right\}, \]
and denote the $L^2$-oscillation of the matrix $A$ as
\begin{equation}\label{def:oscA}
	\alpha_A(x_0,r) := \inf_{A_0 \in \mathfrak{A}(\Lambda)} \left( \frac{1}{|W_\Omega(x_0,r)|} \iint_{W_\Omega(x_0,r)} |A(x,t) - A_0|^2 dxdt \right)^{1/2}, 
\end{equation} 
where, as in Definition \ref{OSCcoef.def} before, the infimum ranges over all constant coefficient matrices. We say $A$ satisfies the weak DKP condition with vanishing trace, if 
\begin{equation}\label{def:muA}
	d\mu_A(x, r) = \alpha_A(x,r)^2 \frac{ dx dr}{r} 
\end{equation} 
is a Carleson measure in $\RR^{n+1}_+$ with vanishing trace (see Definition \ref{def:Carl}).
Let $u$ be a solution to the elliptic equation
\[ -\divg(A(x,t) \nabla u) = 0 \text{ in } \Omega. \]

We consider the flattening map
\[ \Phi: (y,s) \in \RR^{n+1}_+ \mapsto (y,s+\varphi(y)) =: (x,t) \in \Omega, \]
and a function $\tilde{u}: \RR^{n+1}_+ \to \RR$ defined by $\tilde{u}(y,s) := u \circ \Phi(y,s)$. A simple computation shows that $\tilde{u}$ is the solution to the elliptic operator $-\divg(B(y,s) \nabla \tilde{u}) = 0$ in $\RR^{n+1}_+$, where the coefficient matrix $B(y,s)$ is given by
\begin{align}
	B(y,s) & = \det D\Phi  \cdot \left( D\Phi(y,s) \right)^{-1} A(\Phi(y,s)) \left( D\Phi^T(y,s) \right)^{-1} \nonumber \\
	& = \begin{pmatrix}
	\Id_{n} & 0 \\
	\left( -\nabla \varphi(y) \right)^T & 1
\end{pmatrix} A(\Phi(y,s)) \begin{pmatrix}
	\Id_{n} & -\nabla \varphi(y) \\
	0 & 1
\end{pmatrix}. \label{def:B}
\end{align} 
%(If instead, we use the Dahlberg-Kenig-Stein transformation 
%\[ \Phi: (y,s) \in \RR^{n+1}_+ \mapsto (y, cs+\theta_s * \varphi(y)) \in \Omega, \]
%where $\theta_t(\cdot) = \frac{1}{t^n} \theta\left( \frac{\cdot}{t} \right) $ is an approximation of the identity, then 
%\[ D\Phi(y,s) = \begin{pmatrix}
%	\Id_{n} & 0 \\
%	 \vec{b} & d
%\end{pmatrix} \]
%where $\vec{b} = \nabla \theta_s * \varphi(y) = \theta_s * \nabla \varphi(y) $ and $d = c+ \frac{d}{ds} \theta_s * \varphi(y) $; and
%\[ \left(D\Phi(y,s) \right)^{-1} = \begin{pmatrix}
%	\Id_{n} & 0 \\
%	 -\frac{1}{d} \vec{b} & \frac{1}{d}
%\end{pmatrix}. \]
%)
We may define the $L^2$-oscillation of the matrix $B$ as in \eqref{def:oscA}, except to replace the integration region by the corresponding Whitney region in $\RR^{n+1}_+$
%\[ T(x_0,r) = T_{\RR^{n+1}_+}(x_0,r) := \RR^{n+1}_+ \cap B_r(x_0,0). \]
\[ W(x_0, r) := \Delta_r(x_0) \times \left(\frac{r}{2}, r \right]. \]
%Let $M>1$ be a fixed constant whose value is to be determined later. 
Let $A_0$ be a constant coefficient matrix which achieve the infimum for $\alpha_A(x_0, r)$. In particular, $A_0$ has the same constants of ellipticity as $A(\cdot)$. We define
\[ B_0 := \begin{pmatrix}
	\Id_{d-1} & 0 \\
	\left( -\nabla \varphi(x_0) \right)^T & 1
\end{pmatrix} A_0 \begin{pmatrix}
	\Id_{d-1} & -\nabla \varphi(x_0) \\
	0 & 1
\end{pmatrix}. \]
For any $(y,s) \in W(x_0,r)$, by the formula \eqref{def:B} as well as \eqref{eq:moc} we get
\[ \left| B(y,s) - B_0 \right| \lesssim |A(\Phi(y,s)) - A_0| + |\nabla \varphi(y) - \nabla \varphi(x_0) ||A_0| \lesssim |A(\Phi(y,s)) - A_0| + \theta(r), \]
where the constant depends on $\|A(x,t)\|_\infty$. Therefore
\begin{align}
	\left| \alpha_B(x_0,r) \right|^2 & \leq \frac{1}{|W(x_0,r)|} \iint_{W(x_0,r)} |B(y,s) - B_0|^2 dy ds \nonumber \\
	& \lesssim \frac{1}{|W(x_0,r)|} \iint_{W(x_0,r)} |A(\Phi(y,s)) - A_0|^2 dy ds + \theta(r)^2 \nonumber \\
	& \lesssim \frac{1}{|W_\Omega(x_0,r)|} \iint_{W_\Omega(x_0, r)} |A(x,t) - A_0|^2 dx dt + \theta(r)^2 \nonumber \\
	& = \left| \alpha_A(x_0, r) \right|^2 + \theta(r)^2.\label{eq:compCarl}
\end{align}
In the penultimate inequality of \eqref{eq:compCarl}, we use the fact that $\Phi(W(x_0,r)) = W_\Omega(x_0, r)$.
 %the triangle inequality to show $B_r(x_0, 0) \subset B_{Mr} (x_0, \varphi(x_0))$, where the constant $M>1$ depends on the Lipschitz norm $\|\nabla \varphi\|_{L^\infty(B_r(x_0))}$. 
 Similarly to \eqref{def:muA}, we define
\[ d\mu_B(x,r) := \alpha_B(x,r)^2 \frac{dx dr}{r}. \] 
Then we may compute its Carleson norm on each surface ball $\Delta \subset \partial \RR^{n+1}_+$ and
\[ \|\mu_B\|_{\mathcal{C}(\Delta)} \lesssim \|\mu_A\|_{\mathcal{C}(\Delta)} + \int_0^{r_\Delta} \theta(r)^2 \frac{dr}{r}. \]
In particular, if $\theta$ satisfies the square Dini condition, 
%that is, 
%\[ \int_0^* \theta(r)^2 \frac{dr}{r} < +\infty, \]
then
\[ \|\mu_B\|_{\mathcal{C}(\Delta)} \to 0 \text{ as } r_\Delta \to 0. \]
However $\mu_B$ may not be a Carleson measure at large scales because of the extra $\theta(r)^2$ term. To remedy that, let $R>0$ be fixed
%\footnote{For example we may take $R$ to be the size of the coordinate charts for $\Omega$, namely each $\Omega \cap \mathcal{C}(x_i, R) $ is the region above the graph of some function.} 
and we use a similar construction as in Lemma \ref{localizeDKP.lem} to define a new coefficient matrix $\widetilde{B}(\cdot)$, so that $\widetilde{B} \equiv B$ in $\Gamma(0,R/2)$, and $\widetilde{B}(\cdot)$ is a constant coefficient matrix in $\RR^{n+1}_+ \setminus \Gamma(0,R)$.
To be more precise, let $\widetilde{B}_0$ be a constant matrix which achieves the minimum of $\gamma(0, 100cR)$ for the matrix $B(\cdot)$. We define
\[ \widetilde{B}(y,s) = \mathbbm{1}_{\Gamma(0,R)}(y,s) \left[(1- f(|y|))B(y,s) + f(|y|)\widetilde{B}_0 \right] +\mathbbm{1}_{\left(\Gamma(0,R)\right)^c}(y,s) \widetilde{B}_0,  \]
where $f$ is a piece-wise linear function defined as in \eqref{def:f}.
Then Lemma \ref{localizeDKP.lem} implies $\widetilde{B}(\cdot)$ is indeed a Carleson measure with vanishing trace in $\RR^{n+1}_+$.

Let $\omega_\Omega^{X_0}$ denote the elliptic measures corresponding to the matrix $A$ in $\Omega$. Let $\omega^{Y_0}$ and $ \widetilde{\omega}^{Y_0}$ denote the elliptic measure corresponding to the matrix $B$ and $ \widetilde{B}$, respectively, in $\RR^{n+1}_+$. Theorem \ref{mainvmo.thrm} gives that $\widetilde{\omega}^{Y_0} \ll \mathcal{L}^n = dx$ and the Poisson kernel $\widetilde{k}(x) := \frac{d\widetilde{\omega}^{Y_0}}{dx}(x)$ satisfies $\log \widetilde{k} \in VMO_{loc}(\RR^n)$. Similar to the discussions before \eqref{eq:logVMO}, this implies that $\omega^{Y_0} \ll \mathcal{L}^n$ in $B(0, R/2)$, and moreover the Poisson kernel $k(x) = \frac{d\omega^{Y_0}}{dx}(x)$ satisfies $\log k \in VMO(\RR^n \cap B(0, R/2))$.

%Let $x \in \overline{B}_R$ be arbitrary. 
A simple change of variable shows that
\begin{equation}\label{eq:covhm}
	\omega_\Omega^{X_0} (B_r(x,\varphi(x))) = \omega^{\Phi^{-1}(X_0)}(\Phi^{-1}(B_r(x,\varphi(x)))). 
\end{equation} 
Besides, for each $x \in \overline{B(0, R/2)}$ there exists a constant $M>1$ which only depend on $\|\nabla \varphi\|_{L^\infty(\overline{B(0,R)})}$ 
%(which is in turn bounded above by $|\nabla \varphi(x)|+\theta(r)$) 
such that
\[ B_{r/M} (x, 0) \subset \Phi^{-1}(B_r(x,\varphi(x))) \subset B_{Mr}(x,0). \]
Let $\Pi_n: \RR^{n+1} \to \partial \RR^{n+1}_+ \approx \RR^n$ denote the projection onto $\RR^n$. Using \eqref{eq:covhm}, the fact that $\partial\Omega$ is a graph, and the Lebesgue differentiation theorem, we have
\begin{align*}
	& \frac{\omega_\Omega^{X_0}(B_r(x,\varphi(x)))}{ \mathcal{H}^n_{\partial\Omega}(B_r(x,\varphi(x)))} \\
	& = \frac{\omega^{\Phi^{-1}(X_0)}(\Phi^{-1}(B_r(x,\varphi(x)))) }{\int_{\Pi_n(B_r(x,\varphi(x))) } \sqrt{1+|\nabla \varphi(z)|^2} dz } \\
	& = \dfrac{ \fint_{\Phi^{-1}(B_r(x,\varphi(x))) \cap \partial \RR^{n+1}_+ } k(z) dz  }{ \fint_{\Pi_n B_r(x,\varphi(x)) } \sqrt{1+|\nabla \varphi(z)|^2} dz} \cdot \frac{\mathcal{L}^n(\Phi^{-1}(B_r(x,\varphi(x))) \cap \partial \RR^{n+1}_+) }{\mathcal{L}^n(\Pi_n B_r(x,\varphi(x))) } \\
	& \to \frac{k(x)}{\sqrt{1+|\nabla \varphi(x)|^2}} \quad \text{ as } r\to 0+.
\end{align*}
Therefore $\omega_\Omega^{X_0} \ll \mathcal{H}^n_{\partial\Omega}$ and the corresponding Poisson kernel in $\Omega$
\[ k_\Omega(x,\varphi(x)) := \frac{ d\omega_\Omega^{X_0}}{ d\mathcal{H}^n_{\pO}}(x,\varphi(x)) \] satisfies
\[ k_\Omega(x,\varphi(x)) = \frac{k(x)}{ \sqrt{1+|\nabla \varphi(x)|^2}}. \]
Since $\sqrt{1+|\nabla \varphi(x)|^2}$ is continuous and (locally) bounded above and below, it follows that $\log k_\Omega \in VMO_{loc}(\partial\Omega \cap B(0, R/3))$. Therefore we have proven Theorem \ref{thm:squareDini}.

\section*{Declarations}
\textbf{Funding}.
T.T. was partially supported by the Craig McKibben \& Sarah Merner Professor in Mathematics, and by NSF grant DMS-1954545. Z.Z. was partially supported by NSF grant DMS-1902756.

\textbf{Conflict of interest}.
On behalf of all authors, the corresponding author states that there is no conflict of interest.

\textbf{Availability of data and material}. Not applicable.

\textbf{Code availability}. Not applicable.

\bibliographystyle{alpha}
\bibdata{references}
\bibliography{references}

\end{document}